\documentclass{article}
\usepackage{amsmath}
\usepackage{amsfonts}
\usepackage{amsthm}
\usepackage{amssymb, setspace}
\usepackage{mathtools}
\usepackage{lipsum}
\usepackage{color,graphicx,epsfig,geometry,fancyhdr,hyperref}
\usepackage[T1]{fontenc}
\usepackage{float}
\usepackage{subfigure} 
\usepackage{commath}
\usepackage{bbm}
\usepackage{mathrsfs,fleqn}
\usepackage{pgfplots}
\pgfplotsset{compat=newest}
\newtheorem{theorem}{Theorem}[section]

\newcommand{\N}{\mathbb{N}}

\newcommand{\R}{\mathbb{R}}
\newcommand{\C}{\mathbb{C}}

\newcommand{\dnu}{\partial_\nu}

\newcommand{\dnuz}{\partial_{\nu_z}}
\newcommand{\pdr}{\partial_r}
\newcommand{\dd}{\mathrm{d}}

\newcommand{\exD}{\mathbb{R}^2 \setminus \overline{D}}

\usepackage{comment}
\usepackage{ulem}
\numberwithin{equation}{section}

\begin{document}

\begin{flushleft}
\Large 
\noindent{\bf \Large Factorization method for a simply supported obstacle from point source measurements via far--field transformation}

\end{flushleft}

\vspace{0.2in}

{\bf  \large Isaac Harris\footnote{corresponding author's email: \texttt{harri814@purdue.edu}} }\\
\indent {\small Department of Mathematics, Purdue University, West Lafayette, IN 47907, United States of America }\\
\indent {\small Email: \texttt{harri814@purdue.edu} }\\

{\bf  \large Andreas Kleefeld}\\
\indent {\small Forschungszentrum J\"{u}lich GmbH, J\"{u}lich Supercomputing Centre, } \\
\indent {\small Wilhelm-Johnen-Stra{\ss}e, 52425 J\"{u}lich, Germany}\\
\indent {\small University of Applied Sciences Aachen, Faculty of Medical Engineering and } \\
\indent {\small Technomathematics, Heinrich-Mu\ss{}mann-Str. 1, 52428 J\"{u}lich, Germany}\\
\indent {\small Email: \texttt{a.kleefeld@fz-juelich.de}}\\

\begin{abstract}
\noindent We consider an inverse shape problem for recovering an unknown simply supported obstacle in two dimensions from near--field point--source measurements for the biharmonic Helmholtz equation. The measured data consist of the scattered field and its Laplacian on a closed measurement curve surrounding the obstacle. By exploiting an operator splitting of the biharmonic operator, we decouple the scattered field into propagating and evanescent components. This decoupling allows us to reformulate the measured data in terms of an acoustic near--field operator for a sound--soft scatterer. Since the acoustic near--field operator does not directly admit the symmetric factorization required by the factorization method, we introduce a far--field transformation (defined independently of the obstacle) that augments the near--field operator into a far--field operator with a symmetric factorization. This yields a rigorous factorization method characterization of the obstacle and leads to a practical reconstruction algorithm based on spectral data of the transformed operator. Finally, we present numerical experiments with synthetic data that demonstrate stable reconstructions under noise and illustrate the role of regularization, including a variant that uses only the scattered field data.
\end{abstract}



\section{Introduction}
In this paper, we consider the inverse shape problem of recovering a simply supported obstacle in a thin elastic plate from measured near--field data. This two--dimensional scattering problem concerns a biharmonic Helmholtz equation coupled with simply supported boundary conditions. We develop a factorization method for reconstructing the unknown obstacle from near--field measurements generated by a point--source incident field. The factorization method is one of many well--known qualitative methods. Qualitative inversion methods are non--iterative imaging techniques used to reconstruct the shape, size, and location of an unknown scatterer directly from the measured data. This is done by evaluating an imaging functional derived from the scattering data, without solving a full optimization problem for the boundary. The problem we consider here corresponds to the scattering of flexural waves in thin elastic plates, which is commonly modeled by the biharmonic Helmholtz equation. This model has applications ranging from nondestructive testing to structural health monitoring \cite{applicationref1,applicationref2,applicationref3,applicationref4}. For our flexural wave scattering problem, we assume that the scatterer is simply supported with Poisson ratio (defined by the thickness and frequency, see, e.g., \cite{ffdata-simpsupport}) equal to one. This corresponds to the time--harmonic out--of--plane displacement field satisfying a fourth--order partial differential equation. Here, the presence of an unknown simply supported obstacle, denoted by a bounded open set $D \subset \R^2$, leads to a measurable scattered field that we can use to reconstruct the region $D$. Recovering an unknown scatterer from near--field measurements has been studied for many acoustic scattering problems, e.g., in \cite{nf-fft-dsm,HarrisRome,Hu-nfFM,shixu-nfFM,QinColton}. Also, for some recent work on applying the factorization method to elastic scatterers, see, e.g., \cite{ElasticFM1,ElasticFM2}.

The approach we propose is based on the factorization method (see the monograph \cite{kirschbook}), which provides a rigorous characterization of the obstacle in terms of spectral data of an associated compact operator. Indeed, the factorization method allows one to connect the unknown obstacle with the range of the data operator. Therefore, the inversion algorithm only requires determining whether a given function depending on a `sampling point' $z\in \R^2$ is in the range of the measured data operator. Since our measurements are collected in the near--field generated by point sources, two difficulties arise: (i) the biharmonic scattering problem is fourth order with little or no literature on extending the factorization method to this model, and (ii) the associated near--field operator does not directly yield the symmetric factorization required by the factorization method. Our analysis relies on an operator splitting of the forward problem, which decouples the biharmonic equation into a Helmholtz equation and a modified Helmholtz equation. This allows us to reformulate the measured data to correspond to acoustic near--field data for a sound--soft scatterer. With this, following the far--field transformation ideas in \cite{nf-fft-dsm}, we introduce explicit bounded operators (defined independently of the scatterer $D$) that augment the biharmonic near--field operator into the associated far--field operator for a sound--soft acoustic scatterer, to which the results in \cite{kirschbook} can be applied.

The main contributions of this paper are as follows. We establish well--posedness of the simply supported biharmonic scattering problem, where the incident field is assumed to be a point source, by exploiting the decoupling of the scattered field into propagating and evanescent parts. This is similar to studying scattering in a waveguide. With this, we can derive a far--field transformation for the biharmonic near--field operator associated with our problem that yields a symmetric factorization. This will then provide a rigorous factorization method for our biharmonic inverse scattering problem. To the best of our knowledge, this is the first time the factorization method has been applied to an inverse flexural wave scattering problem with near--field measurements. In \cite{FM-wRafaGen} the factorization method was applied to a penetrable scatterer using far--field measurements and in \cite{near-lsmBH,DSM-BH24,LSM-BH25} other qualitative inversion methods were applied to a clamped obstacle.

The paper is organized as follows. In Section \ref{dp}, we formulate the direct scattering problem and prove well--posedness. Next, we study the inverse shape problem, introduce the near--field operator, and derive the far--field transformation leading to the factorization method given in Section \ref{ip}. In Section \ref{numerics}, we describe the numerical implementation and provide reconstructions. Here we provide a detailed numerical study of the inversion algorithm studied in the previous section. We conclude in Section \ref{conclusion}, where we discuss our numerical findings as well as other possible areas to explore for this problem.

\section{The Direct Scattering Problem}\label{dp}
In this section, we introduce the direct scattering problem under consideration and discuss the well--posedness. Here, we will show well--posedness for the biharmonic scattering of a simply supported obstacle by a point source incident wave. With this in mind, we assume that the simply supported obstacle is modeled as a bounded domain $D \subset \R^2$ with a $C^2$--smooth boundary $\partial D$. To illuminate the obstacle, we assume that an incident point source is used, which we denote as 
$$u^{\text{inc}}( \cdot  \, , y) = \mathbb{G} ( \cdot  \, , y),\quad  \text{where $y \in \Gamma$} \,\,\,  \text{such that} \,\, \, \overline{D} \subset\,\text{Int}(\Gamma)$$
for some smooth closed curve $\Gamma$. 
Here, $\Gamma$ is the measurement curve where the sources and receivers are assumed to be placed. For our problem, the incident point source is given by
\begin{align}\label{fundamentalsol}
\mathbb{G}(x,y) = - \frac{\text{i}}{8k^2} \left[ H^{(1)}_0(k|x-y|) -  H^{(1)}_0(\text{i}k|x-y|)\right] \quad \text{for $x\neq y$}
\end{align}
with $H^{(1)}_0$ being defined as the Hankel function of the first kind of order zero and $k>0$ is the wave number. Notice that $\mathbb{G}( \cdot \, , y)$ satisfies 
$$\Delta^2 \mathbb{G} ( \cdot  \, , y) - k^4\mathbb{G} ( \cdot  \, , y) = -\delta( \cdot - y) \quad \text{ in $\R^2$}$$
along with the radiation conditions
$$\lim_{r \to \infty} \sqrt{r} \big(\pdr{\mathbb{G}(x,y)} - \text{i} k \mathbb{G}(x,y)\big) = 0 \quad \text{ and } \quad \lim_{r \to \infty} \sqrt{r}\big(\pdr{ \Delta \mathbb{G}(x,y)} - \text{i} k \Delta \mathbb{G}(x,y)\big) = 0$$
with $r=|x|$. The above limits are assumed to hold uniformly for $\hat{x}=x/|x|$ and $y \in \Gamma$. This gives that $\mathbb{G}( \cdot \, , y)$ is the radiating fundamental solution to the biharmonic Helmholtz operator $(\Delta^2 -k^4)$ in $\R^2$. Note that by \eqref{fundamentalsol} we have that $\mathbb{G}(x,y)$ is $C^\infty$--smooth for all $x \neq y$ in $\R^2$.

With our simply supported obstacle assumption, we have that the total displacement field and its Laplacian have zero trace on the boundary $\partial D$. This would imply that the corresponding scattered field $u^{\text{scat}}(x,y)$ for $x \in \exD$ and $y \in \Gamma$ satisfies the biharmonic scattering problem for a fixed wave number $k>0$, given by 
\begin{align}
\Delta^2 u^{\text{scat}} - k^4 u^{\text{scat}} = 0  \quad\quad\quad\quad& \text{in } \mathbb{R}^2\setminus \overline{D}, \label{biharmonic} \\
u^{\text{scat}} =-u^{\text{inc}}  \quad \text{ and } \quad \Delta{u^{\text{scat}}} =- \Delta{u^{\text{inc}}} \quad & \text{on $\partial D$}. \label{bcs}
\end{align}
We let the derivatives in \eqref{biharmonic}--\eqref{bcs} be with respect to the variable $x$. To close the system, the scattered field $u^{\text{scat}}$ satisfies the radiation conditions (see, e.g., \cite{iterative-BH23,invsource-biharm})
\begin{align}\label{SRC}
\lim_{r \to \infty} \sqrt{r}(\pdr{u^{\text{scat}}} - \text{i} k u^{\text{scat}}) = 0 \quad \text{ and } \quad \lim_{r \to \infty} \sqrt{r}(\pdr{ \Delta u^{\text{scat}}} - \text{i} k \Delta u^{\text{scat}}) = 0  
\end{align}
which is assumed to hold uniformly in $\hat{x}=x/|x|$. We will show that there is a unique scattered field $u^{\text{scat}}( \cdot  \, ,y) \in H^2_{loc}(\exD)$ that satisfies the scattering problem \eqref{biharmonic}--\eqref{SRC} for any $y \in \Gamma$. 

The direct problem studied here is similar to those studied in \cite{biwellposed,clamped-wellposed}. Here, our analysis of well--posedness is significantly different and will aid in our study of the inverse problem. We also note that this reduces the regularity of the boundary to $C^2$ from $C^{3,\alpha}$, which is needed in \cite{clamped-wellposed}. To this end, we show that the direct problem \eqref{biharmonic}--\eqref{SRC} can be decoupled into two second--order boundary value problems. Indeed, it is common to express the scattered field $u^{\text{scat}}$ in terms of two auxiliary functions denoted $u_\textrm{pr}$ and $u_\textrm{ev}$ that are defined by the expressions 
\begin{align}\label{vhvm}
u_\textrm{pr} = -\frac{1}{2k^2}\big(\Delta u^{\text{scat}} - k^2 u^{\text{scat}}\big) \quad \text{and} \quad  u_\textrm{ev} = \frac{1}{2k^2}\big(\Delta u^{\text{scat}} + k^2 u^{\text{scat}}\big).
\end{align}
This would imply that 
$$u^{\text{scat}} = u_\textrm{pr} + u_\textrm{ev}$$
just as in \cite{DSM-BH24} as well as other works. Here, as in other works, we have that $u_\textrm{pr}$ is the propagating part of the scattered field whereas $u_\textrm{ev}$ is the evanescent part of the scattered field.

Using the splitting of the biharmonic Helmholtz operator
$$\Delta^2 - k^4 = (\Delta-k^2)(\Delta+k^2)=(\Delta + k^2)(\Delta - k^2)$$
one obtains that $u_\textrm{pr}$ satisfies the Helmholtz equation and $u_\textrm{ev}$ satisfies the modified Helmholtz equation, i.e., with wave number $\mathrm{i}k$. Now, we let the simply supported boundary condition \eqref{bcs} be given by 
$$u^{\text{scat}} = f  \quad \text{ and } \quad \Delta{u^{\text{scat}}} =g \quad \text{on $\partial D$}$$
for some given $f$ and $g \in H^{3/2}(\partial D)$. By the above decomposition we have that equation \eqref{bcs} becomes 
$$u_\textrm{pr} + u_\textrm{ev} = f  \quad \text{ and } \quad  -k^2 u_\textrm{pr} + k^2 u_\textrm{ev} = g \quad \text{on $\partial D$}.$$
Here, we have used the fact that 
$$\Delta(u_\textrm{pr} + u_\textrm{ev}  ) = -k^2 u_\textrm{pr} + k^2 u_\textrm{ev} \quad \text{in } \exD$$ 
and then solved for $u_\textrm{pr}$ and $u_\textrm{ev} $ in the boundary condition. With this, we see that the scattering problem \eqref{biharmonic}--\eqref{SRC} decouples to 
\begin{align}
\Delta u_\textrm{pr} + k^2 u_\textrm{pr} = 0\quad \text{in } \exD \quad \text{with } \quad   u_\textrm{pr} = \frac{k^2 f - g}{2k^2} \quad \text{on } \partial D 
\label{vhvmeq1}  
\end{align}
and 
\begin{align}
\Delta u_\textrm{ev} - k^2 u_\textrm{ev} = 0\quad \text{in } \exD \quad \text{with } \quad   u_\textrm{ev} = \frac{k^2 f + g}{2k^2} \quad \text{on } \partial D \label{vhvmeq2}  
\end{align}
along with the radiation conditions 
\begin{align}
\lim_{r \to \infty} \sqrt{r}(\pdr{u_\textrm{pr}} - \text{i}ku_\textrm{pr}) = 0 \quad \text{ and }  \quad   \lim_{r \to \infty} \sqrt{r}(\pdr{u_\textrm{ev}} - \text{i}ku_\textrm{ev}) = 0. \label{vhvmeq3}
\end{align}
It is known that $u_\textrm{ev} $ and $\partial_r u_\textrm{ev} $ decay exponentially as $r \to \infty$ (see, e.g., \cite{DongLi24,DSM-BH24}). 

Since we have a decoupled system of boundary value problems for $u_\textrm{pr}$ and $u_\textrm{ev}$ given in \eqref{vhvmeq1}--\eqref{vhvmeq3} we can now prove well--posedness of \eqref{biharmonic}--\eqref{SRC}. Indeed, we first show uniqueness, so assume that $u^{\text{scat}}$ satisfies \eqref{biharmonic}--\eqref{SRC} with $u^{\text{inc}}$ replaced by zero. This would imply that $f=g=0$ in \eqref{vhvmeq2}. It is clear from the analysis in \cite[Chapter 5]{Cakoni-Colton-book} that this would imply that $u_\textrm{pr}=u_\textrm{ev}=0$ in $\exD$ which implies that \eqref{biharmonic}--\eqref{SRC} would admit at most one solution.

Now we prove existence of the solution to  \eqref{biharmonic}--\eqref{SRC} and the stability estimate. To this end, by again appealing to \cite[Chapter 5]{Cakoni-Colton-book} we have that  
$$\| u_\textrm{pr} \|_{H^1(B_{3R} \setminus \overline{D})} \leq C \left\{\| f \|_{H^{3/2}(\partial{D})}  + \| g \|_{H^{3/2}(\partial{D})}  \right\}$$ 
and 
$$\| u_\textrm{ev} \|_{H^1(B_{3R} \setminus \overline{D})} \leq C \left\{\| f \|_{H^{3/2}(\partial{D})}  + \| g \|_{H^{3/2}(\partial{D})}  \right\}$$
where $B_R$ is the disk centered at the origin of radius $R>0$ such that $\overline{D} \subset B_R$. To obtain the stability estimate, we notice that by elliptic regularity (see, e.g., \cite[Section 8.6]{Salsa}) we have that $u_\textrm{pr}$ and $u_\textrm{ev}$ are in $H^2_{loc}(\exD)$ since we have assumed that $f$ and $g \in H^{3/2}(\partial D)$. Indeed, by appealing to an interior elliptic regularity estimate in $B_{2R} \setminus \overline{B_R}$ we have that 
$$\| u_\textrm{pr} \|_{H^2(B_{2R} \setminus \overline{B_R})} \leq C \| u_\textrm{pr} \|_{{L^2}(B_{3R} \setminus \overline{D})}  
\quad \text{ and } \quad 
\| u_\textrm{ev} \|_{H^2(B_{2R} \setminus \overline{B_R})} \leq C \| u_\textrm{ev} \|_{{L^2}(B_{3R} \setminus \overline{D})}.$$ 
From this, we have the 
estimate 
$$\| u_\textrm{pr} \|_{H^2(B_{2R} \setminus \overline{B_R})} + \| u_\textrm{ev} \|_{H^2(B_{2R} \setminus \overline{B_R})}  \leq C \left\{\| f \|_{H^{3/2}(\partial{D})}  + \| g \|_{H^{3/2}(\partial{D})}  \right\}.$$ 
Now, we notice that the scattered field $u^{\text{scat}} = u_\textrm{pr} + u_\textrm{ev}$ satisfies 
$$ \Delta u^{\text{scat}} = -k^2(u_\textrm{pr} - u_\textrm{ev}) \quad \text{on $\exD$}$$  
along with the simply supported boundary conditions. Similarly, by appealing to the global elliptic regularity estimate in $B_R \setminus \overline{D}$ one obtains that 
\begin{align*} 
\| u^{\text{scat}} \|_{H^2(B_R \setminus \overline{D})} &\leq C \Big\{ \| u_\textrm{pr} - u_\textrm{ev} \|_{L^2(B_R \setminus \overline{D})}  \\
& \hspace{1in}+ \| u^{\text{scat}} \|_{L^2(B_R \setminus \overline{D})} + \| u^{\text{scat}} \|_{H^{3/2}(\partial{D})} + \| u^{\text{scat}} \|_{H^{3/2}(\partial{B_R})}  \Big\}.
\end{align*}
By appealing to the previous estimates, we have that 
$$ \| u_\textrm{pr} - u_\textrm{ev} \|_{L^2(B_R \setminus \overline{D})}+ \| u^{\text{scat}} \|_{L^2(B_R \setminus \overline{D})} \leq C \left\{\| f \|_{H^{3/2}(\partial{D})}  + \| g \|_{H^{3/2}(\partial{D})}  \right\}.$$ 
This gives the estimate  
$$\| u^{\text{scat}} \|_{H^2(B_R \setminus \overline{D})} \leq C \left\{\| f \|_{H^{3/2}(\partial{D})}  + \| g \|_{H^{3/2}(\partial{D})} +  \| u_\textrm{pr} +u_\textrm{ev} \|_{H^{3/2}(\partial{B_R})}  \right\},$$ 
where we have also used the fact that $u^{\text{scat}} = f$ on $\partial D$. We can now use the Trace Theorem and the previous estimates to obtain 
\begin{align*}
\| u_\textrm{pr} +u_\textrm{ev} \|_{H^{3/2}(\partial{B_R})}  &\leq   C \left\{\| u_\textrm{pr} \|_{H^2(B_{2R} \setminus \overline{B_R})} + \| u_\textrm{ev} \|_{H^2(B_{2R} \setminus \overline{B_R})} \right\} \\
&\leq   C \left\{\| f \|_{H^{3/2}(\partial{D})}  + \| g \|_{H^{3/2}(\partial{D})}  \right\}
\end{align*}
which gives the final estimate 
$$\| u^{\mathrm{scat}} \|_{H^2(B_R \setminus \overline{D})} \leq C \left\{\|f \|_{H^{3/2}(\partial{D})}  + \|g \|_{H^{3/2}(\partial{D})}  \right\}.$$ 
Therefore, we have that the direct scattering problem \eqref{biharmonic}--\eqref{SRC} is well--posed since $u^{\text{inc}}( \cdot \, , y)$ and $\Delta u^{\text{inc}}( \cdot \, , y) \in H^{3/2}(\partial D)$ for any $y \in \Gamma$. This gives the following result. 

\begin{theorem} \label{dp-wellposed}
The biharmonic scattering problem for a simply supported obstacle by a point source incident wave \eqref{biharmonic}--\eqref{SRC} has a unique solution $u^{\mathrm{scat}}( \cdot  \, ,y) \in H^2_{loc}(\exD)$ for all $y \in \Gamma$ that satisfies the estimate 
$$\| u^{\mathrm{scat}}( \cdot  \, ,y) \|_{H^2(B_R \setminus \overline{D})} \leq C \left\{\| u^{\mathrm{inc}}( \cdot \, , y) \|_{H^{3/2}(\partial{D})}  + \| \Delta u^{\mathrm{inc}}( \cdot \, , y) \|_{H^{3/2}(\partial{D})}  \right\}$$ 
for some constant $C>0$. 
\end{theorem}

Now that we have well--posedness of the direct scattering problem \eqref{biharmonic}--\eqref{SRC}, we turn our attention to solving the inverse shape problem. Therefore, in the following sections we will assume that 
$$u^{\text{scat}}( x , y) \quad \text{ and } \quad \Delta_x u^{\text{scat}}( x , y) \quad \text{ is known for all $x,y \in \Gamma$.}$$ 
From this measured scattering data we will develop a numerical method for recovering the scatterer $D$. Note that a similar problem was considered in \cite{iterative-BH23}, where the scatterer was a clamped obstacle and the point sources are located at a finite number of points in $\exD$. Here, we will derive a factorization method for this problem by appealing to similar analysis as in \cite{nf-fft-dsm}.

\section{Solution to the Inverse Shape Problem}\label{ip}
In this section, we study the inverse shape problem of recovering the scatterer from measured near--field data. Here, we propose using a factorization method \cite{kirschbook}. Recently, in \cite{near-lsmBH} the linear sampling method was studied for recovering a clamped obstacle. In our analysis, we apply a far--field transform to our near--field operator, as is done in \cite{nf-fft-dsm}. This is useful for multiple reasons. First, in acoustic scattering it is known that the near--field operator does not produce a symmetric factorization that is needed to apply the theoretical results of the factorization method, as in \cite{HarrisRome,Hu-nfFM,shixu-nfFM}. Moreover, we see from \cite{DSM-BH24,LSM-BH25} that the far--field data from a biharmonic scattering problem can be handled similarly to the acoustic case, where there is a wealth of literature to exploit. 

Now, recall that by Theorem \ref{dp-wellposed} we have that the simply supported obstacle scattering problem given by \eqref{biharmonic}--\eqref{SRC} is well--posed for any source point $y \in \Gamma$. With this in mind, we will assume that the following near--field data 
\begin{align} \label{nf-data}
\big\{ u^{\text{scat}}( x , y) \quad \text{ and } \quad \Delta_x u^{\text{scat}}( x , y) \quad : \quad \text{ for all $x,y \in \Gamma$} \big\}
\end{align}
is known. Therefore, we study the inverse shape problem of recovering the unknown scatterer $D$ from the data \eqref{nf-data}. Since we have the measured scattered field and its Laplacian for all $x,y \in \Gamma$, we can now define the near--field operator 
\begin{align} \label{nf-op}
N: L^2(\Gamma) \longrightarrow L^2(\Gamma) \quad \text{ given by } \quad (N g)(x) = \int_{\Gamma} \left[ \Delta_x u^{\text{scat}}( x , y) -k^2 u^{\text{scat}}( x , y) \right] g(y) \, \dd s(y). 
\end{align}
We use this near--field operator to exploit the decoupled system \eqref{vhvmeq1}--\eqref{vhvmeq3} used to study $u_\textrm{pr}$ and $u_\textrm{ev}$ in the previous section. Indeed, by \eqref{vhvm} we see that 
$$(N g)(x) = -2k^2 \int_{\Gamma} u_\textrm{pr}(x,y) \, g(y) \, \dd s(y).$$
This implies that the scattering data is equivalent to the near--field data for an acoustic sound--soft obstacle.

\subsection{Uniqueness of the Inverse Problem}\label{unique}
Before studying the factorization method using a far--field transformation for this problem, we will first consider the uniqueness of the inverse problem. Therefore, we will show that the near--field data given in \eqref{nf-data} uniquely determines the scatterer. Here, we notice that by the definition of the fundamental solution $\mathbb{G}(x,y)$ we have that $u_\textrm{pr}(x,y)$ satisfies 
\begin{align} \label{vh-bvp} 
\Delta_x u_\textrm{pr} + k^2 u_\textrm{pr} = 0 \quad \text{in } \,  \exD \quad \text{and } \quad u_\textrm{pr} = \frac{1}{2k^2} \Phi_k ( \cdot \, , y)  \quad \text{on } \,  \partial D
\end{align}
along with the radiation condition \eqref{vhvmeq3}. In equation \eqref{vh-bvp} the boundary data $\Phi_k ( \cdot \, , y)$ is the corresponding fundamental solution for the Helmholtz equation given by 
\begin{align} \label{Phi} 
\Phi_k ( x , y) = \frac{\text{i}}{4} H^{(1)}_0(k|x-y|)  \quad \text{for $x\neq y$.}
\end{align}
Therefore, we have that $-2k^2 u_\textrm{pr} ( x , y)$ is the radiating acoustic scattered field for a sound--soft obstacle generated by an acoustic point source. 

To proceed, we first show that $u_\textrm{pr}$ satisfies the standard reciprocity identity 
\begin{align}\label{recip}
u_\textrm{pr}(x,y) =  u_\textrm{pr}(y,x) \quad \text{ for all $x,y \in \Gamma$}. 
\end{align}
Indeed, we can prove the claim just as it is done in \cite{QinColton}. By appealing to Green's representation formula we have that for $x,y \in \Gamma$
$$u_\textrm{pr}(x,y)= \int_{\partial D} u_\textrm{pr} ( \cdot \, , y ) \dnu  \Phi_k( x , \cdot)  - \dnu u_\textrm{pr} ( \cdot \, , y ) \Phi_k( x , \cdot) \, \dd s$$
and similarly 
$$u_\textrm{pr}(y,x)= \int_{\partial D}  v_H ( \cdot \, , x ) \dnu  \Phi_k( y , \cdot)  -  \dnu u_\textrm{pr} ( \cdot \, , x ) \Phi_k( y , \cdot) \, \dd s.$$
We can now use the fact that $u_\textrm{pr} ( \cdot \, , y )$ and $u_\textrm{pr} ( \cdot \, , x )$ are radiating solutions to the Helmholtz equation in $\exD$ to obtain 
$$0= \int_{\partial D} \dnu u_\textrm{pr} ( \cdot \, , y ) u_\textrm{pr}( \cdot \, , x) - u_\textrm{pr} ( \cdot \, , y ) \dnu  u_\textrm{pr}( \cdot \, , x) \, \dd s$$
and similarly by using Green's second identity in $D$ we can get the equality 
$$0= \int_{\partial D} \dnu \Phi_k ( \cdot \, , y ) \Phi_k ( x , \cdot) - \Phi_k ( \cdot \, , y ) \dnu  \Phi_k ( x , \cdot)  \, \dd s.$$
From the boundary condition in \eqref{vh-bvp} we have that  
\begin{align*}
0=& \int_{\partial D} \left(2k^2 u_\textrm{pr} ( \cdot \, , y ) - \Phi_k ( \cdot \, , y )  \right) \dnu \left(2k^2 u_\textrm{pr} ( \cdot \, , x ) - \Phi_k ( \cdot \, , x )  \right) \\
&\hspace{1.5in}   -  \dnu \left(2k^2 u_\textrm{pr} ( \cdot \, , y ) - \Phi_k ( \cdot \, , y )  \right) \left(2k^2 u_\textrm{pr} ( \cdot \, , x ) - \Phi_k ( \cdot \, , x )  \right) \, \dd s \\
& =- 2k^2 \left[u_\textrm{pr} ( x , y ) - u_\textrm{pr} ( y,x )  \right]
\end{align*}
by expanding the terms in the integrand and using the previous integral identities. Note that we have also used the symmetry of $\Phi_k$. With this identity, we can now prove uniqueness of the inverse problem given the near--field data \eqref{nf-data}.

\begin{theorem} \label{ip-uniqueness}
The near--field data defined by equation \eqref{nf-data} uniquely determines the scatterer. 
\end{theorem}
\begin{proof}
To prove the claim we proceed by way of contradiction. To this end, assume there are two scatterers $D^{(1)}$ and $D^{(2)}$ that produce the same near--field data set in \eqref{nf-data}. This would imply that the corresponding $u_\textrm{pr}^{(1)} ( x , y )$ for $D^{(1)}$ and $u_\textrm{pr}^{(2)} ( x , y )$ for $D^{(2)}$ coincide for all $x,y \in \Gamma$. Since the exterior Dirichlet problem for the Helmholtz equation is well--posed this implies that 
$$u_\textrm{pr}^{(1)} ( x , y ) = u_\textrm{pr}^{(2)} ( x , y ) \quad \text{ for all $y\in \Gamma$ and $x \in \R^2 \setminus \overline{(D^{(1)} \cup D^{(2)})}$}$$
by the unique continuation principle. Now, due to the reciprocity identity in \eqref{recip} we can see that 
$$u_\textrm{pr}^{(1)} ( x , y ) = u_\textrm{pr}^{(2)} ( x , y ) \quad \text{ for all $x,y \in \R^2 \setminus \overline{(D^{(1)} \cup D^{(2)})}$}.$$

Without loss of generality, assume that $\overline{D^{(1)}}$ is not contained in $\overline{D^{(2)}}$, then $\partial D^{(1)} \setminus \overline{D^{(2)}}$ is a non--empty set. Therefore, let $y^* \in \partial D^{(1)} \setminus \overline{D^{(2)}}$ and define the sequence 
$$y_n = y^* + \frac{1}{n} \nu_1(y^*)  \quad \text{for $n \in \N$} $$ 
where $\nu_1$ is the outward unit normal vector for $\partial D^{(1)}$. For $n$ sufficiently large we can assume that the sequence $y_n \in \R^2 \setminus \overline{(D^{(1)} \cup D^{(2)})}$. So by the reciprocity identity we have that
$$u_\textrm{pr}^{(1)} ( x , y_n ) = u_\textrm{pr}^{(2)} ( x , y_n ) \quad \text{ for all $x \in \R^2 \setminus \overline{(D^{(1)} \cup D^{(2)})}$}.$$
This gives that by letting $x \to y^*$ we can obtain
$$ u_\textrm{pr}^{(1)} ( y^* , y_n ) = u_\textrm{pr}^{(2)} ( y^* , y_n )  \quad \text{ and by the boundary condition \eqref{vh-bvp}} \quad \frac{1}{2k^2} \Phi_k ( y^* , y_n ) = u_\textrm{pr}^{(2)} ( y^* , y_n )$$
since $y^* \in \partial D^{(1)}$. This gives a contradiction since 
$$ |\Phi_k ( y^* , y_n )| \to \infty  \quad \text{ whereas } \quad |u_\textrm{pr}^{(2)} ( y^* , y_n )| < \infty  \quad \text{ as $n \to \infty$ }$$
by elliptic regularity. This proves the claim. 
\end{proof}

We have proven that our inverse problem has the uniqueness property. This was done by exploiting the decoupling of the biharmonic scattering problem \eqref{biharmonic}--\eqref{SRC} into boundary value problems for $u_\textrm{pr}$ and $u_\textrm{ev}$ given in \eqref{vhvmeq1}--\eqref{vhvmeq3}. The main idea is to use the fact that the near--field data defined in \eqref{nf-data} is equivalent to the knowledge of $u_\textrm{pr}(x,y)$ for all $x,y \in \Gamma$ (i.e. the propagating part of the scattered field). This allowed us to work just as in the acoustic scattering case.

\subsection{Factorization Method via Far--Field Transform}\label{FM}
Now that we know our near--field data \eqref{nf-data} uniquely determines the scatterer $D$, we turn our attention to deriving a method for the reconstruction of the scatterer. Just as in the previous sections, our analysis relies heavily on the decoupling of the forward problem. Here, we derive a factorization method to recover the scatterer from the near--field operator \eqref{nf-op}. The factorization method was initially developed in \cite{firstFM} to recover both sound--soft and sound--hard obstacles from far--field data. See, e.g., \cite{fmconstant,FM-wave,FM-multifreq,regfm2,Hu-nfFM,FMgibc-obstacale} and the references therein for the factorization method applied to other acoustic scattering problems. For our problem, one of the main novelties is that we apply the factorization method to the biharmonic scattering problem \eqref{biharmonic}--\eqref{SRC}, which has yet to be rigorously studied for this or other biharmonic scattering problems. 

To begin, we first note that just as before we can exploit the decoupling of the forward problem to apply techniques that are well--established for acoustic scatterers. Since we have near--field data, we have to do some pre--processing of the near--field operator just as in \cite{nf-fft-dsm,Hu-nfFM} to obtain a symmetric factorization. Here, we transform our near--field operator into a far--field operator, motivated by the analysis in \cite{nf-fft-dsm}. Recall that we have the near--field operator 
$$(N g)(x) = \int_{\Gamma} \left[ \Delta_x u^{\text{scat}}( x , y) -k^2 u^{\text{scat}}( x , y) \right] g(y) \, \dd s(y)$$
maps $L^2(\Gamma)$ into itself. This would imply that 
$$(N g)(x) =  -2k^2 \int_{\Gamma} u_\textrm{pr} ( x , y) \, g(y) \, \dd s(y).$$
by equation \eqref{vhvm}. Thus, we notice that our biharmonic near--field operator is equivalent to the acoustic near--field operator for scattering by a sound--soft obstacle generated by an acoustic point source. This is due to the fact that $-2k^2 u_\textrm{pr}$ is the solution to the corresponding acoustic scattering problem by \eqref{vh-bvp}.

In order to derive a factorization for the near--field operator defined in \eqref{nf-op}, we can appeal to the results in \cite{nf-fft-dsm}. Therefore, we define the single layer boundary integral operator 
$$SL_{\partial D}: H^{-1/2}(\partial D) \longrightarrow H^1_{\text{loc}} (\R^2 \setminus\overline{D}) \quad \text{ given by } \quad \big(SL_{\partial D}\big) \varphi= \int_{\partial D} \Phi_k( \cdot \, ,  z ) \varphi (z) \, \text{d}s(z)$$
along with the associated inverse of the single layer potential   
$$S^{-1}_{\partial D \to \partial D} : H^{1/2}(\partial D) \longrightarrow H^{-1/2}(\partial D) \quad \text{such that } \quad  \big(S_{\partial D \to \partial D} \big) \varphi =\int_{\partial D} \Phi_k ( \cdot \, ,  z) \varphi(z) \, \text{d}s(z) \Big|_{\partial D}.$$
Here, we again let $\Phi_k( \cdot \, ,  z ) $ denote the fundamental solution for the Helmholtz equation. For convenience, we will let 
\begin{align}\label{T-operator}
T= - S^{-1}_{\partial D \to \partial D} \quad \text{as a mapping} \quad H^{1/2}(\partial D) \longrightarrow H^{-1/2}(\partial D). 
\end{align} 
By \cite[Lemma 1.14]{kirschbook} we have that $T$ is a well-defined bounded linear operator provided that $k^2$ is not a Dirichlet eigenvalue of the negative Laplacian in $D$. Now, we define the bounded linear operator 
\begin{align}\label{sl-op1}
S : L^2(\Gamma)  \longrightarrow H^{1/2}(\partial D) \subset L^2(\partial D) \quad  \text{given by} \quad   S g = \int_{\Gamma}  \Phi_k( \cdot \, , y)  g(y) \, \text{d}s(y) \Big|_{\partial D} 
\end{align}
and its associated dual--operator 
\begin{align}\label{sl-op2}
S^{\top} : L^2(\partial D)  \longrightarrow H^{1/2}(\Gamma) \subset L^2(\Gamma) \quad  \text{ given by} \quad S^{\top} \varphi = \int_{\partial D}  \Phi_k( z , \cdot)  \varphi(z) \, \text{d}s(z) \Big|_{\Gamma}
\end{align}
with respect to the bilinear $L^2$ dual--product. Since $-2k^2 u_\textrm{pr} ( \cdot \, , y)$ satisfies equations (1)--(2) in \cite{nf-fft-dsm}, we have that 
$$(N g)(x) =  -2k^2 \int_{\Gamma} u_\textrm{pr} ( x , y) g(y) \, \dd s(y)  \quad \text{ has the factorization } \quad N = S^\top T S$$
by appealing to equation (8) in \cite{nf-fft-dsm}.

The theoretical foundation for the factorization method requires a symmetric factorization of the near--field operator. As is the case with other near--field operators, we do not have this for our problem since we have $S^\top$ (i.e., the dual--operator) rather than $S^*$ (i.e., the adjoint operator) in our factorization. In order to avoid this, \cite{nf-fft-dsm} shows that the aforementioned near--field operator can be augmented to have a symmetric factorization. To this end, we define the Dirichlet--to--Far--Field operator   
\begin{align}
\mathcal{Q}: H^{1/2}(\Gamma) \longrightarrow L^2(\mathbb{S}^{1})  \quad  \text{ given by} \quad (\mathcal{Q} f)(\hat{x}) = w^{\infty}(\hat{x}), \quad \text{ for any } \; \hat{x} \in \mathbb{S}^{1}, \label{Q-operator}
\end{align}
where $w\in H^1_{\text{loc}}(\R^2 \setminus \overline{\text{Int}(\Gamma)})$ is the unique solution to
\begin{align}
\Delta w +k^2 w  =  0 \quad \text{in}  \quad \R^2 \setminus \overline{\text{Int}(\Gamma)} \quad \text{ with } \quad w|_\Gamma = f \label{eq-ext1}
\end{align} 
along with the Sommerfeld radiation condition, where $\mathbb{S}^{1}$ denotes the unit `sphere' in $\R^2$. It is well known that the far--field pattern $w^{\infty}$ exists and that $w$ satisfies the asymptotic expansion 
$$w(x)= \frac{ \mathrm{e}^{\mathrm{i}\pi/4} }{ \sqrt{8 \pi k} }  \cdot  \frac{\text{e}^{\text{i}k|x|}}{\sqrt{ |x| } } \left\{w^{\infty}(\hat{x}) + \mathcal{O} \left( \frac{1}{|x|}\right) \right\}\; \textrm{  as  } \;  |x| \to \infty,$$
where again $\hat x=x/|x|$ just as in \cite{nf-fft-dsm}. 

For completeness, we show how to use the operator $\mathcal{Q}$ to augment the near--field operator associated with our biharmonic scattering problem. Therefore, by following the calculations originally derived in \cite{nf-fft-dsm}, we have that 
$$(\mathcal{Q} \, S^{\top} \varphi) (\hat{x}) = \int_{\partial D} \text{e}^{-\text{i}k \hat{x}\cdot z} \varphi(z) \,  \text{d}s(z)  \quad  \textrm{for any } \quad \varphi \in L^2(\partial D),$$
where we have used the asymptotics of the fundamental solution $\Phi_k( \cdot \, ,  z )$. This corresponds to the adjoint of the Herglotz operator 
\begin{align}
H: L^2( \mathbb{S}^{1}) \longrightarrow L^2( \partial D)   \quad  \text{is given by } \quad  (H g)(z) = \int_{\mathbb{S}^{1}} \text{e}^{\text{i}k  z \cdot \hat{x}} g(\hat{x})  \,  \mathrm{d} s(\hat{x}) \Big|_{\partial D}, \label{H-operator}
\end{align}
where its adjoint is defined as 
$$(H^* \varphi) (\hat{x}) =  \int_{\partial D} \text{e}^{-\text{i}k \hat{x}\cdot z} \varphi(z)  \,   \text{d}s(z)  \quad  \textrm{for any } \quad \varphi \in L^2(\partial D).$$ 
This implies that $\mathcal{Q} \, S^{\top} = H^*$ and by taking the transpose of this expression, we obtain the equality $S\mathcal{Q}^\top = (H^*)^\top$ where 
$$\left( (H^*)^\top g \right) (z) = \int_{\mathbb{S}^{1}} \text{e}^{- \text{i}k  z \cdot \hat{x}} g(\hat{x})  \, \mathrm{d} s(\hat{x}) \Big|_{\partial D} .$$
To continue, in an attempt to derive a symmetric factorization we see that 
\begin{align*}
\left( (H^*)^\top g \right) (z) = \int_{\mathbb{S}^{1}} \text{e}^{- \text{i}k  z \cdot \hat{x}} g(\hat{x})  \,  \mathrm{d} s(\hat{x}) \Big|_{\partial D}= \int_{\mathbb{S}^{1}} \text{e}^{\text{i}k  z \cdot \hat{x}} g(-\hat{x})  \,  \mathrm{d} s(\hat{x}) \Big|_{\partial D}
\end{align*}
which implies that $(H^*)^\top =  H\mathcal{R}$, where the operator 
\begin{align}
\mathcal{R}:  L^2( \mathbb{S}^{1}) \longrightarrow L^2(\mathbb{S}^{1})\quad  \text{is given by } \quad (\mathcal{R}g)(\hat{x}) = g(-\hat{x}). \label{R-operator}
\end{align}
We see that by the definition \eqref{R-operator} that $\mathcal{R}=\mathcal{R}^{-1}$. By appealing to the factorization $N = S^\top T S$ and the above calculations we have the following result. 

\begin{theorem} 
The near--field operator defined by equation \eqref{nf-op} associated with the biharmonic scattering problem  \eqref{biharmonic}--\eqref{SRC} satisfies 
\begin{align} \label{FFT}
\mathcal{Q} N \mathcal{Q}^\top \mathcal{R} =H^*TH \quad \text{ such that } \quad \mathcal{Q} N \mathcal{Q}^\top \mathcal{R} :  L^2( \mathbb{S}^{1}) \longrightarrow L^2(\mathbb{S}^{1}),
\end{align}
where the bounded linear operators $T$, $\mathcal{Q}$, $H$, and $\mathcal{R}$ are defined in \eqref{T-operator}, \eqref{Q-operator},  \eqref{H-operator}, and \eqref{R-operator} respectively, provided that $k^2$ is not a Dirichlet eigenvalue of the negative Laplacian in $D$.
\end{theorem}

With this result, we see from \eqref{FFT} that the augmented operator $\mathcal{Q} N \mathcal{Q}^\top \mathcal{R}$ does indeed have a symmetric factorization. This implies that this post--processing can be used to symmetrize the factorization and apply the results found in \cite{firstFM,kirschbook}. Note that this is similar to the idea presented in \cite{Hu-nfFM}, where a so--called outgoing-to-incoming operator is used to symmetrize the factorization of the near--field operator associated with acoustic problems. Notice that the augmented operator $\mathcal{Q} N \mathcal{Q}^\top \mathcal{R}$ is the far--field operator for a sound--soft acoustic scatterer by \cite[equation (1.55)]{kirschbook}.

Now, we can show that the factorization method can be applied to our inverse problem. We recall that using the bounded linear operators $\mathcal{Q}$ in \eqref{Q-operator} and $\mathcal{R}$ in \eqref{R-operator}, our augmented operator $\mathcal{Q} N \mathcal{Q}^\top \mathcal{R}$ is the far--field operator for a sound--soft acoustic scatterer. Therefore, we can use \cite[Theorem 1.25]{kirschbook} to obtain the factorization method for recovering the simply supported obstacle.

\begin{theorem} 
Assume that $k^2$ is not a Dirichlet eigenvalue of the negative Laplacian in $D$. Then we have that 
\begin{align} \label{FMresult}
z \in D \quad \text{if and only if } \quad \sum\limits_{j=1}^{\infty} \frac{1}{|\lambda_j| }  \left| \left( \mathrm{e}^{-\mathrm{i}k\hat{x}\cdot z} \, , \, \psi_j(\hat{x}) \right)_{L^2(\mathbb{S}^1)} \right|^2 < \infty 
\end{align}
with $(\lambda_j , \psi_j) \in \C_{\neq 0} \times L^2(\mathbb{S}^1)$ the eigenvalues and vectors for the compact normal operator $\mathcal{Q} N \mathcal{Q}^\top \mathcal{R}$.
\end{theorem}

This result is given by deriving a range identity for the obstacle and using Picard's criteria to connect the obstacle with the spectral data for $\mathcal{Q} N \mathcal{Q}^\top \mathcal{R}$. From equation \eqref{FMresult} we see that the augmented operator $\mathcal{Q} N \mathcal{Q}^\top \mathcal{R}$ can be used to recover the region $D$. To this end, notice that by \eqref{FMresult} this result would imply that 
\begin{align} \label{FMnumeric}
z \in D \quad \text{if and only if } \quad W(z) = \left[\sum\limits_{j=1}^{\infty} \frac{1 }{|\lambda_j| }  \left| \left( \mathrm{e}^{-\mathrm{i}k\hat{x}\cdot z} \, , \, \psi_j(\hat{x}) \right)_{L^2(\mathbb{S}^1)} \right|^2 \right]^{-1} >0. 
\end{align}
Again, we assume that $(\lambda_j , \psi_j) \in \C_{\neq 0} \times L^2(\mathbb{S}^1)$ are the eigenvalues and vectors for the operator $\mathcal{Q} N \mathcal{Q}^\top \mathcal{R}$. Since the near--field operator $N$ is known and the operators $\mathcal{Q}$ and $\mathcal{R}$ can be precomputed, $W(z)$ defined in \eqref{FMnumeric} is computable from the data. Therefore, from the near--field data we can compute $W(z)$ and provide the contour plot to recover $D$. 

Notice that since $N$ is compact, this would imply that $\mathcal{Q} N \mathcal{Q}^\top \mathcal{R}$ is compact as well. This would imply that the eigenvalues $|\lambda_j| \to 0$ rapidly as $j \to \infty$. By the definition of the imaging functional $W(z)$, we see that dividing by the eigenvalues could cause numerical instabilities when trying to recover $D$. In order to ensure stability as well as a rigorous characterization of the scatterer, researchers have worked to add regularization to the factorization method; see, e.g., \cite{GLSM,regfm2,RegFM}. In particular, if we assume that the measured near--field operator $N^\delta$ has added random error such that 
$$ \| N^\delta - N \|_{\mathcal{L}( L^2(\Gamma))} = \mathcal{O}(\delta) \quad \text{implies} \quad  \|\mathcal{Q}  (N^\delta - N)\mathcal{Q}^\top \mathcal{R} \|_{\mathcal{L}( L^2(\mathbb{S}^1) )} = \mathcal{O}(\delta)  \quad \text{as $\delta \to 0$.}$$ 
This would be the case in many physical applications, where the measured data is polluted with random noise. 

In order to have a stable reconstruction with noisy near--field data, the main result in \cite{regfm2} implies that if we define the regularized imaging functional 
\begin{align} \label{FMnumeric2}
W_{\text{Reg}}(z;\alpha,\delta) = \left[\sum\limits_{j=1}^{\infty}  \frac{ \phi^2 \big(|\lambda^\delta_j| , \alpha \big)}{|\lambda^\delta_j| } \left|\left( \mathrm{e}^{-\mathrm{i}k\hat{x}\cdot z} \, , \, \psi^\delta_j(\hat{x}) \right)_{L^2(\mathbb{S}^1)} \right|^2 \right]^{-1} ,
\end{align}
where $(\lambda^\delta_j , \psi^\delta_j) \in \C_{\neq 0} \times L^2(\mathbb{S}^1)$ are the eigenvalues and vectors for the perturbed operator $\mathcal{Q} N^\delta \mathcal{Q}^\top \mathcal{R}$, then we have that 
$$z \in D \quad \text{if and only if } \quad \liminf\limits_{\alpha \to 0^+}\liminf\limits_{\delta \to 0^+} W_{\text{Reg}}(z;\alpha,\delta)>0.$$
Here, it is assumed that $\phi (t, \alpha ) $ is a non--negative continuous filter function derived from a regularization scheme such that 
\begin{align} \label{filter-assumptions}
\lim\limits_{\alpha \to 0} \phi(t, \alpha ) = 1, \quad \phi(t , \alpha ) \leq C_{\text{Reg}}  \quad \text{ and } \quad  \phi(t, \alpha ) \leq C_{\alpha} t \quad \text{for all} \,\, t, \alpha>0,
\end{align}
where the constant $C_{\text{Reg}}$ is independent of the regularization parameter $\alpha$. 
An example of such a filter function comes from Tikhonov regularization where 
$$\phi (t, \alpha ) =  \frac{t^2}{t^2+\alpha} \quad \text{ for all } \quad t,\alpha>0.$$ 
By appealing to the continuity of the spectrum it can be shown that for any $z$ 
$$\liminf\limits_{\alpha \to 0^+}\liminf\limits_{\delta \to 0^+} W_{\text{Reg}}(z;\alpha,\delta) >0  \quad \text{if and only if} \quad W(z)>0 .$$ 
So in the case of noisy near--field measurements, it would be optimal to use the regularized imaging functional $W_{\text{Reg}}(z;\alpha,\delta)$ defined in \eqref{FMnumeric2} to recover the scatterer. Most recently, this regularized variant of the factorization method we used in \cite{GMR} to recover a strain--gradient interface in a thin elastic body. 

One observation is that the operators $\mathcal{Q}$ and $\mathcal{R}$ are defined without the knowledge of $D$. This implies that they can be computed without any a priori knowledge of the scatterer. Indeed, for the operator 
$$(\mathcal{R}g)(\hat{x}) = g(-\hat{x})$$
we first note that 
$$\hat x = \big(\cos (\theta) \, , \, \sin (\theta) \big) \quad \text{ for } \quad \theta \in [0, 2\pi).$$
Therefore, we can consider $g$ as given by  
$$g(\hat x) = g\big( (\cos\theta \, , \, \sin \theta) \big)$$ 
which can be viewed as a function of the angular variable, i.e. $g(\theta)$. Notice that the sum of angles formulas gives 
$$- \hat x =  \big(\cos (\theta+\pi) \, , \, \sin (\theta + \pi) \big) \quad \text{ for } \quad \theta \in [0, 2\pi).$$
With this, we obtain that $(\mathcal{R}g)(\theta) = g(\theta+\pi)$ and by some simple calculations  
$$(\mathcal{R} g)(\theta) = \int\limits_0^{2\pi} R_{func} (\theta, \phi) g(\phi)  \,  \text{d}{\phi} \quad \text{where} \quad R_{func}(\theta, \phi) = \frac{1}{2\pi} \sum_{|m| = 0}^{\infty} \text{e}^{\text{i}m(\theta - \phi + \pi)}$$
by using a Fourier series to express $g$ as is done in \cite{nf-fft-dsm}. 

Now, in order to pre--compute operator $\mathcal{Q}$, we will assume that $\Gamma=\partial B_R$ where $B_R$ is the ball centered at the origin of radius $R>0$. Similar to the above, we have that 
$$ x \in \Gamma \quad \text{ is given by } \quad x=R (\cos \theta \, , \, \sin \theta) \quad \text{ for } \quad \theta \in [0, 2\pi).$$
Therefore, the corresponding Dirichlet data in \eqref{eq-ext1} can be seen as a function depending on the angular variable such that 
$$f(\theta) = f\big( R (\cos\theta \, , \, \sin \theta) \big) \quad \text{ for } \quad \theta \in [0, 2\pi).$$
By using separation of variables to find an analytical solution for \eqref{eq-ext1} (see \cite{postdocpaper} for details) we have that the Dirichlet--to--Far--Field operator can be written as 
$$ (\mathcal{Q} f)(\theta) = \int\limits_0^{2\pi} Q_{func} (\theta, \phi) f(\phi)  \,  \text{d}{\phi} \quad \text{where} \quad Q_{func} (\theta, \phi) = \frac{ 2 }{\pi \text{i}} \sum_{|m| = 0}^{\infty} \frac{\text{e}^{\text{i}m(\theta - \phi - \pi/2)}}{H^{(1)}_{m}(kR)}$$
with $H^{(1)}_{m}$ being the first kind Hankel function of order $m$. Even though the kernel functions used to defined $\mathcal{Q}$ and $\mathcal{R}$ are given as an infinite series, it has been shown in \cite{nf-fft-dsm,postdocpaper} that these operators can be approximated well by a truncated series.

{\color{black} {\bf Remark:} For the case when $\Gamma \neq \partial B_R$ we can define Dirichlet--to--Far-Field operator $\mathcal{Q}$ via integral equations. Indeed, we can make the ansatz that the scattered field in \eqref{eq-ext1} can be written as a double--layer integral operator 
$$DL_{\Gamma}: H^{1/2}(\Gamma) \longrightarrow H^1_{\text{loc}} \big(\R^2 \setminus\overline{\text{Int}(\Gamma}) \big)\quad \text{ given by } \quad \big(DL_{\Gamma}\big) \psi = \int_{\Gamma} \dnuz \Phi_k( \cdot \, ,  z ) \psi (z) \, \text{d}s(z).$$
In order to satisfy the boundary condition on $\Gamma$ in \eqref{eq-ext1}, we have that the potential $\psi$ satisfies
$$(I+ 2D_{\Gamma \to \Gamma})\psi=2f \quad \text{ which implies that } \quad \psi = 2(I+ 2D_{\Gamma \to \Gamma})^{-1}f$$
provided that $-k^2$ is not a Neumann eigenvalue of $\Delta$ in Int$(\Gamma)$. Here, we have used the jump relation for the integral operator and 
$$(D_{\Gamma \to \Gamma})\psi = \int_{\Gamma} \dnuz \Phi_k( \cdot \, ,  z ) \psi (z) \, \text{d}s(z)\Big|_{\Gamma}.$$
We refer to \cite{coltonkress,mclean} for in--depth analysis on the boundary operator discussed here. Now, as in \cite{DSM-BH24} we have that the solution to \eqref{eq-ext1} is given by 
$$ 2\big(DL_{\Gamma}\big)(I+ 2D_{\Gamma \to \Gamma})^{-1} f \quad \text{ which implies that } \quad \mathcal{Q} f= 2\big(DL^{\infty}_{\Gamma}\big)(I+ 2D_{\Gamma \to \Gamma})^{-1}f$$
where 
$$\big(DL^{\infty}_{\Gamma}\big) \psi =  \int_{\Gamma} \dnuz \text{e}^{-\text{i}k\hat{x}\cdot z}  \psi(z) \, \text{d}s(z)= -\text{i}k \int_{\Gamma} \nu(z)\cdot \hat{x} \text{e}^{-\text{i}k\hat{x}\cdot z}  \psi(z) \, \text{d}s(z)$$ 
is the far--field pattern for the double--layer integral operator approach. We will also note that there are other ways to define the 
Dirichlet--to--Far-Field operator $\mathcal{Q}$ as well as the transformation of the near--field data to far--field data for acoustic scatterers \cite[Lemma 2.13]{kirschbook}.

\section{Numerical Experiments}\label{numerics}
In this section, we will provide numerical reconstructions using the regularized imaging functional defined in \eqref{FMnumeric2}. We will see that this method provides reconstructions of simply supported obstacles from noisy near--field data. From our numerical experiments we see that the post--processing via the far--field transformation leads to a stable indicator for recovering the scatterer. To this end, we first generate synthetic data by solving the direct scattering problem using boundary integral equations. Then from the approximated solution we can form a discrete approximation of the near--field operator \eqref{nf-op}. By then computing the eigenvalue decomposition of the discrete operator $\mathcal{Q}N^\delta\mathcal{Q}^\top\mathcal{R}$ we can evaluate imaging functional $W_{\mathrm{Reg}}(z)$ given in equation \eqref{FMnumeric2} to produce reconstructions of $D$.

\subsection{Boundary Integral Equations for the Scattering Problem}
To begin, we need to approximate the solution to the scattering problem \eqref{biharmonic}--\eqref{SRC}. Recall that the boundary conditions for the scattering problem is given by 
$$u^{\text{scat}}=-\mathbb{G}( \cdot \, ,y) \quad \text{and} \quad  \Delta u^{\text{scat}}=-\Delta \mathbb{G}(\cdot \, ,y) \quad \text{for $y\in \Gamma$.}$$
Here $\Gamma$ is some given boundary curve surrounding the scatterer $D$. The fundamental solution for the biharmonic Helmholtz equation $\mathbb{G}(x ,y)$ is again given by 
$$\mathbb{G}(x,y) = -\frac{1}{2k^2}\left[\Phi_k(x,y)-\Phi_{\mathrm{i}k}(x,y)\right]$$
with $\Phi_\tau (x,y)=\frac{\mathrm{i}}{4}H^{(1)}_0(\tau|x - y|)$ denoting the fundamental solution of the Helmholtz equation with wave number $\tau$.
This now implies that 
$$\Delta_x \mathbb{G}(x,y) = \frac{1}{2}\left[\Phi_k(x,y)+\Phi_{\mathrm{i}k}(x,y)\right]$$
Using operator splitting discussed in Section \ref{dp}, we have $u^{\text{scat}}=u_{\text{pr}}+u_{\text{ev}}$, where $u_{\text{pr}}$ solves the Helmholtz equation outside $D$ and $u_{\text{ev}}$ solves the modified Helmholtz equation outside $D$.

Using a single--layer ansatz (similar to \cite{clampedTEexist}), gives
\begin{eqnarray}
 u_{\text{pr}}(x)&=&\mathrm{SL}_{k}\psi(x)\,,\quad x\in\mathbb{R}^2 \backslash \overline{D}\,,\label{start1}\\
 u_{\text{ev}}(x)&=&\mathrm{SL}_{\mathrm{i}k}\varphi(x)\,,\quad x\in\mathbb{R}^2 \backslash \overline{D}\,,
 \label{start2}
\end{eqnarray}
where the single layer operator is defined by
\begin{eqnarray*}
\mathrm{SL}_{k}\phi(x)&=&\int_{\partial D}\Phi_k(x,y)\phi(y)\,\mathrm{d}s(y)\,,\quad  x\notin \partial D\,,
\end{eqnarray*}
where as above $\Phi_\tau (x,y)$ denotes the fundamental solution of the Helmholtz equation in two dimensions with wave number $\tau$. Since both single--layer integral operators satisfy  the PDEs and radiation conditions, we only need to satisfy the boundary conditions to approximate the scattered field. Therefore, letting $x$ approach the boundary in (\ref{start1}) and (\ref{start2}) along with the jump conditions (see \cite[Theorem 3.1]{coltonkress}) and using the first boundary condition, yields the first integral equation 
\begin{eqnarray}
 \mathrm{S}_{k}\psi(x)+\mathrm{S}_{\mathrm{i}k}\varphi(x)=-\mathbb{G}(x,y)\,,\quad x\in\partial D\,,\label{start1a}
\end{eqnarray}
where the single--layer boundary integral operator is defined by
\begin{eqnarray*}
    \mathrm{S}_{k}\phi(x)&=&\int_{\partial D}\Phi_k(x,y)\phi(y)\,\mathrm{d}s\,,\quad x\in \partial D\,.
\end{eqnarray*}
Notice that, since we have
$$\Delta \mathrm{SL}_{k}\psi(x)=-k^2 \mathrm{SL}_{k}\psi(x) \quad \text{and} \quad  \Delta \mathrm{SL}_{\mathrm{i}k}\psi(x) = k^2 \mathrm{SL}_{\mathrm{i}k}\psi(x) $$
for all $x\in\mathbb{R}^2 \backslash \overline{D}$, we can write with the second boundary condition
\begin{eqnarray}
-k^2 \mathrm{S}_{k}\psi+k^2 \mathrm{S}_{\mathrm{i}k}\psi=-\Delta \mathbb{G}(x,y) \,,\quad x\in\partial D\,,\label{start1b}
\end{eqnarray}
Combining \eqref{start1a}--\eqref{start1b} yields the following system of boundary integral equations
\begin{eqnarray*}
 \begin{pmatrix}
 \mathrm{S}_{k} & \mathrm{S}_{\mathrm{i}k}\\
 -k^2 \mathrm{S}_{k} & k^2 \mathrm{S}_{\mathrm{i}k}
\end{pmatrix}
 \begin{pmatrix}
 \psi\\
 \varphi
 \end{pmatrix}
 =-
 \begin{pmatrix}
 \mathbb{G}(\cdot \, ,y)\\
 \Delta \mathbb{G}(\cdot \, ,y)
 \end{pmatrix}\,
\end{eqnarray*}
on the boundary of the scatterer.

The above system of boundary integral equations can be solved explicitly. We have
\begin{eqnarray}\psi=\frac{1}{2k^2}S_k^{-1}\Phi_k(\cdot ,y) 
\label{los1}
\end{eqnarray}
\text{ and } 
\begin{eqnarray}
\varphi=-\frac{1}{2k^2}S_{\mathrm{i}k}^{-1}\Phi_{\mathrm{i}k}(\cdot ,y)\,.
\label{los2}
\end{eqnarray}
Note that both $S_k^{-1}$ and $S_{\mathrm{i}k}^{-1}$ exist since $k^2$ and $-k^2$ are assumed not to be Dirichlet eigenvalues of $-\Delta$ in $D$. After we solve this system for $\psi$ and $\varphi$, we can insert this into our ansatz to compute $u^{\text{scat}}$ on $\Gamma$
\begin{eqnarray*}
u^{\text{scat}}(x,y)&=&\mathrm{SL}_{k}\psi(x)+\mathrm{SL}_{\mathrm{i}k}\varphi(x)\\
&=&\frac{1}{2k^2}\left(\mathrm{SL}_{k}S_k^{-1}\Phi_k(x,y)-\mathrm{SL}_{\mathrm{i}k}S_{\mathrm{i}k}^{-1}\Phi_{\mathrm{i}k}(x,y)\right)\,,\quad x\in \Gamma
\end{eqnarray*}
Likewise, we can compute the Laplacian of $u^{\text{scat}}$ on $\Gamma$
\begin{eqnarray*}
\Delta_x u^{\text{scat}}(x,y)&=&-k^2\mathrm{SL}_{k}\psi(x)+k^2\mathrm{SL}_{\mathrm{i}k}\varphi(x)\\
&=&\frac{1}{2}\left(-\mathrm{SL}_{k}S_k^{-1}\Phi_k(x,y)+\mathrm{SL}_{\mathrm{i}k}S_{\mathrm{i}k}^{-1}\Phi_{\mathrm{i}k}(x,y)\right)\,,\quad x\in \Gamma\,.
\end{eqnarray*}
The details of the discretization to compute an approximation to $\psi$ given by \eqref{los1} and $\varphi$ given by \eqref{los2} are given in \cite[Section 4.3]{kleefeldhs}.

\subsection{Numerical Reconstructions}
Now, we are ready to provide some numerical examples that have been implemented in \texttt{MATLAB}. To this end, we will assume that we have the synthetic near--field data associated with the scattering problem \eqref{biharmonic}--\eqref{SRC}. Here, the synthetic near--field data is computed via the boundary integral equations discussed in the previous section. With this, we will take 
$$\Gamma = \partial B_R \,\,\, \text{with $R=3$ unless otherwise stated} \quad  \text{and} \quad k=2$$
for all examples. This implies that the sources and receivers are located at the points 
$$x_j=y_j = R \big( \cos \theta_j \, , \sin \theta_j) \quad \text{for 64 equally spaced points  $\theta_j  \in [0 , 2\pi )$. }$$ 
Therefore, we have that our synthetic noisy near--field data is given by 
$$u^{\text{scat},\delta} (x_i , y_j)= u^{\text{scat}}(x_i , y_j) \big(1 + \delta E_1({i,j}) \big) \quad \text{and} \quad  \Delta_x u^{\text{scat},\delta} (x_i , y_j) =  \Delta_x u^{\text{scat}} (x_i , y_j) \big(1 + \delta E_2({i,j}) \big),$$
where the matrices $E_j$ for $j=1,2$ have random complex--valued entries of size $64 \times 64$ with unit 2--norm i.e. $\|E_j\|_{2} = 1$. Here, we define 
$${\bf N} = \big[ \Delta_x u^{\text{scat},\delta} (x_i , y_j) - k^2 u^{\text{scat},\delta} (x_i , y_j) \big]_{i,j = 1}^{64}$$
which is the discretized version of the near--field operator defined in \eqref{nf-op}.

In order to apply the reconstruction method in \eqref{FMresult}, we need to discretize the operators $\mathcal{Q}$ and $\mathcal{R}$. The discretization of the operators is computed via 
$$ {\bf Q} = \big[ Q_{func} (\theta_{i} , \theta_{j}) \big]_{i,j = 1}^{64} \quad \text{and} \quad {\bf R} = \big[ R_{func} (\theta_{i} , \theta_{j}) \big]_{i,j = 1}^{64}.$$
Note that we define the kernel functions as 
$$Q_{func} (\theta, \phi) = \frac{ 2 }{\pi \text{i}} \sum_{|m| = 0}^{10} \frac{\text{e}^{\text{i}m(\theta - \phi - \pi/2)}}{H^{(1)}_{m}(kR)}  \quad \text{and} \quad R_{func}(\theta, \phi) = \frac{1}{2\pi} \sum_{|m| = 0}^{10} \text{e}^{\text{i}m(\theta - \phi + \pi)}$$
which is the truncated series approximation of the kernel functions used to define the operators in \eqref{Q-operator} and \eqref{R-operator}. With this, we have to compute the imaging functional $W_{\text{Reg}}(z)$ given in \eqref{FMnumeric2} 
$$ {\bf Q} {\bf N} {\bf Q}^\top {\bf R}  \quad \text{and} \quad {\bf b}_z =[ \text{e}^{-\text{i}k z \cdot \hat{x}_1} , \cdots , \text{e}^{-\text{i}k z \cdot \hat{x}_{64}} ]^\top $$
with $\hat{x}_{i} = \big( \cos \theta_{i} \, , \sin \theta_{i} \big)$ with $i=1, \cdots, 64$. Lastly, we consider three different regularization filter functions in the definition of $W_{\text{Reg}}(z)$. For our examples we will use the filter functions 
\begin{align}\label{filters}
 \phi_{\text{Tik}}(t ; \alpha) =  \frac{t^2}{t^2+\alpha}, \,\, \,  \phi_{\text{GLSM}}(t ; \alpha) = \frac{t}{\alpha + t} \,\, \textrm{ and } \,\,  \displaystyle{  \phi_{\text{cut--off}}(t ; \alpha)= \left\{\begin{array}{lr} 1, &  t^2\geq \alpha,  \\
 				&  \\
 0,&  t^2 < \alpha
 \end{array} \right.}  .
\end{align}
These filter functions correspond to Tikhonov regularization, Generalized Linear Sampling Method (GLSM, see \cite{GLSM} for details) and Spectral/SVD cut--off, respectively. 

We can now recover scatterers with the synthetic data. We will consider three scattering regions given by a star, peanut and kite shaped simply supported obstacle. In Table \ref{scatterers}, we give the parameterizations for each of the scatterers under consideration. With this, we will provide numerical examples, where $W_{\text{Reg}}(z)$ is computed via 
$$W_{\text{Reg}}(z)= \left[ \sum\limits_{j=1}^{64} \frac{\phi^2( | \lambda_j  | ; \alpha)}{ | \lambda_j | } \big|{\bf u}_j \cdot {\bf b}_z \big|^2 \right]^{-1}$$
using the built-in \texttt{eig} solver in \texttt{MATLAB} to compute $\lambda_j$ and ${\bf u}_j$, i.e., the eigenvalues and vectors of $ {\bf Q} {\bf N} {\bf Q}^\top {\bf R}$. We use an equally spaced 200$\times$200 sampling grid in the region $[-3,3]^2$, where the imaging functional is normalized to have unit $\infty$--norm. 

\begin{table}[H]
	\centering
	\begin{tabular}{ l l }
		Scatterer     &  Parameterization \\
\hline 
\hline
		\vspace{-2.0ex}\\
		Star--shaped  & $x(t)={\displaystyle 0.25(0.3\cos(5t)+2) \big(\cos(t),\sin(t) \big)^\top}$
		\vspace{1.5ex} \\ 
		Peanut--shaped   & $x(t)={\displaystyle 2\sqrt{0.5\sin(t)^2+0.1\cos(t)^2} \big(\cos(t),\sin(t) \big)^\top}$  
		\vspace{1.5ex}\\ 
		Kite--shaped  & $x(t)={\displaystyle \big(-1.5\sin(t),\cos(t)+0.65\cos(2t)-0.65\big)^\top}$
		\vspace{1.5ex}\\
\hline
\end{tabular}\caption{The boundary parameterizations of $\partial D=x(t)$ for $t\in [0,2\pi)$.} \label{scatterers}
\end{table}

Here, in Figure \ref{fig:scatterers} we provide a visualization of the scatterers defined in Table \ref{scatterers}. We see that the scatterers are non--convex as well as having different sizes. This is to test the strength of the reconstruction algorithm. In our preceding numerical examples, we see that we are able to accurately recover each of the scatterers. We will test how the imaging functional $W_{\text{Reg}}(z)$ depends on the filter function $\phi$, regularization parameter $\alpha$, and added error $\delta$. Also, we will see if we can recover the scatterer only given the measured scattered field (i.e., the Laplacian is not known).

\begin{figure}[h]
	\centering	
	\subfigure[Star--shaped scatterer]{\includegraphics[width=0.32\textwidth]{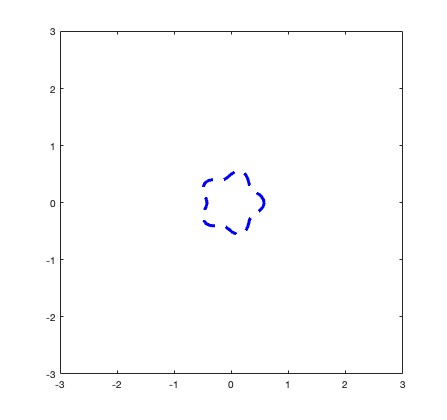}} 
	\subfigure[Peanut--shaped scatterer]{\includegraphics[width=0.32\textwidth]{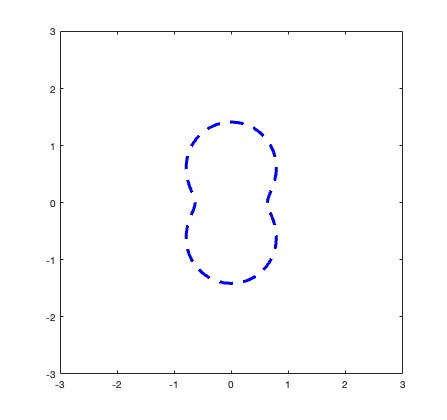}}
	\subfigure[Kite--shaped scatterer]{\includegraphics[width=0.32\textwidth]{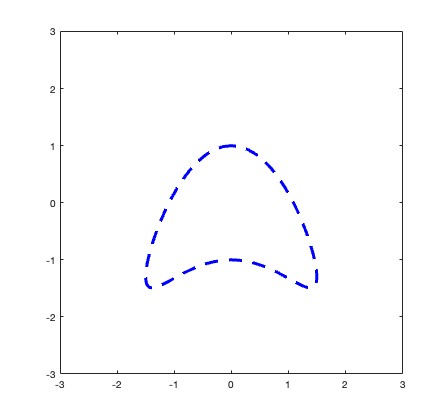}}
	\caption{Visual representations of the three scatterers defined in Table \ref{scatterers}.}\label{fig:scatterers}
\end{figure}

\subsubsection{Reconstructions with and without regularization} 

For this set of examples, we wish to test the accuracy of the reconstruction method with and without regularization. In the case of reconstructions without regularization, we take $\alpha = 0$ in the computed imaging functional $W_{\text{Reg}}(z)$. These examples are contrasted with the reconstruction of the scatterer using the ad--hoc regularization parameter $\alpha=0.0001$ and the Tikhonov filter function given in \eqref{filters}. We see in Figures \ref{recon1} and \ref{recon2} that with added noise in the scattering data, the imaging functional is unable to recover the kite and star--shaped scatterers without regularization. For the kite--shaped scatterer, we let $\delta = 0.05$ for the reconstruction in Figure \ref{recon1}. For the star--shaped scatterer, we let $\delta = 0.02$ in the reconstruction given in Figure \ref{recon2}.

\begin{figure}[h]
\centering 
\includegraphics[scale=0.16]{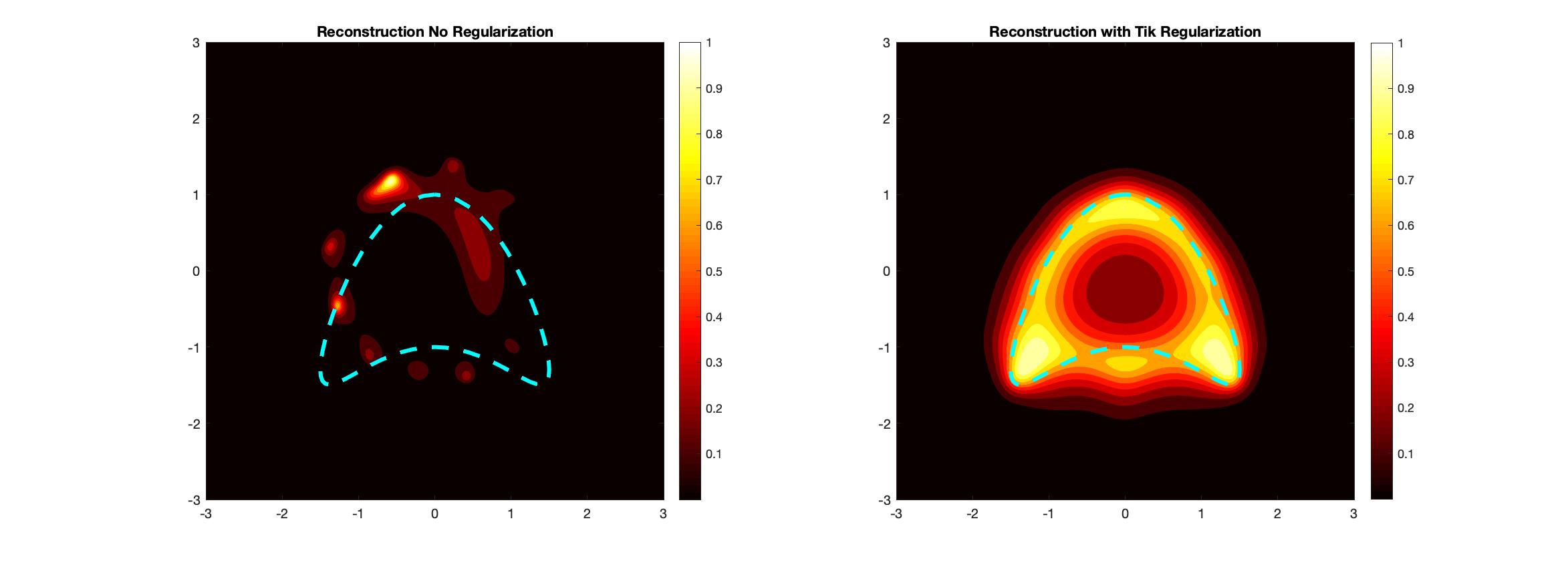}
\caption{Reconstruction of the kite--shaped scatterer with and without regularization where we add $5\%$ random noise to the data. Here we use the Tikhonov filter function given in \eqref{filters} with $\alpha = 0.0001$. Left: reconstruction without regularization and Right: reconstruction with regularization.}
\label{recon1}
\end{figure}

\begin{figure}[h]
\centering 
\includegraphics[scale=0.16]{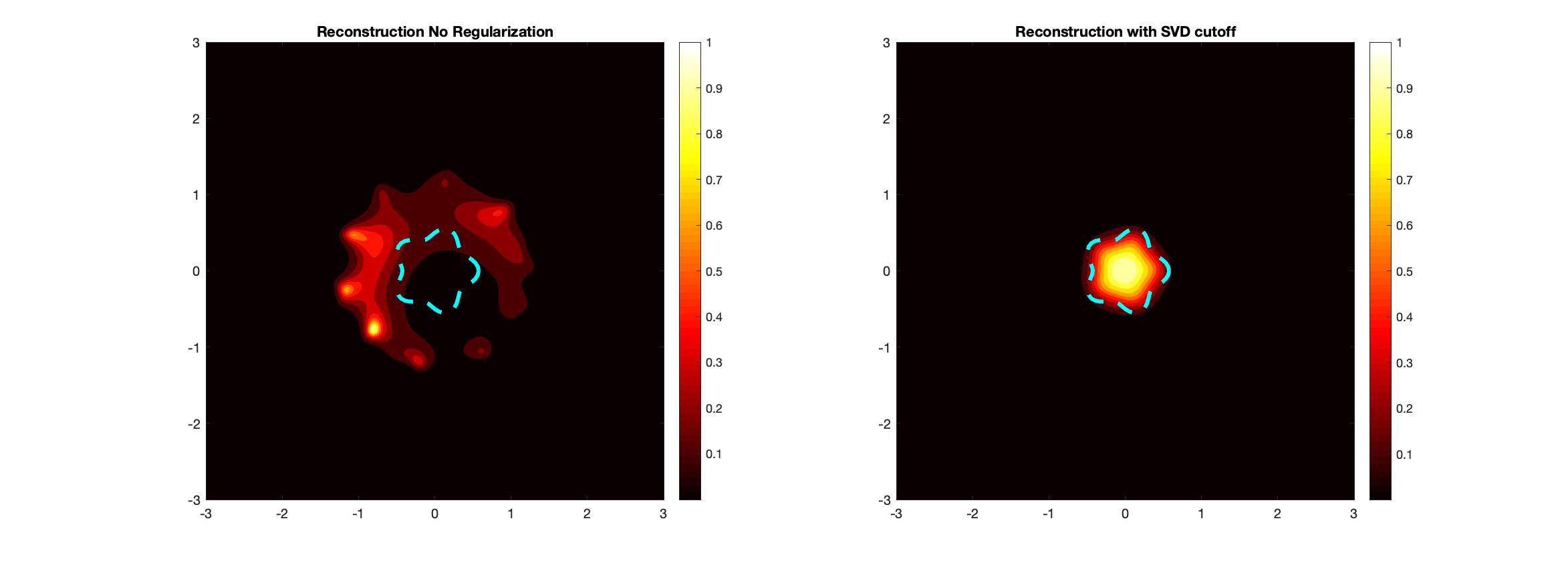}
\caption{Reconstruction of the star--shaped scatterer with and without regularization where we add $2\%$ random noise to the data. Here we use the Tikhonov filter function given in \eqref{filters} with $\alpha = 0.0001$. Left: reconstruction  without regularization and Right: reconstruction with regularization.}
\label{recon2}
\end{figure}

From these two examples, we see that the reconstruction without regularization is unstable when given noisy data. As stated in the previous section, this is due to the fact that without regularization the imaging functional requires you to divide by the eigenvalues of an ill--conditioned matrix (since it approximates a compact operator). This would imply that the regularization is needed for accurate reconstructions. Next, we see if the regularization is necessary when one has accurate data (i.e., with no added error). In Figure \ref{recon3}, we provide the reconstructions of the peanut--shaped scatterer with and without regularization. Here, we again take $\alpha = 0.0001$ for the reconstruction with regularization but we now use the GLSM filter function given in \eqref{filters}. We again see that the regularized imaging functional provides a better reconstruction.

\begin{figure}[h]
\centering 
\includegraphics[scale=0.16]{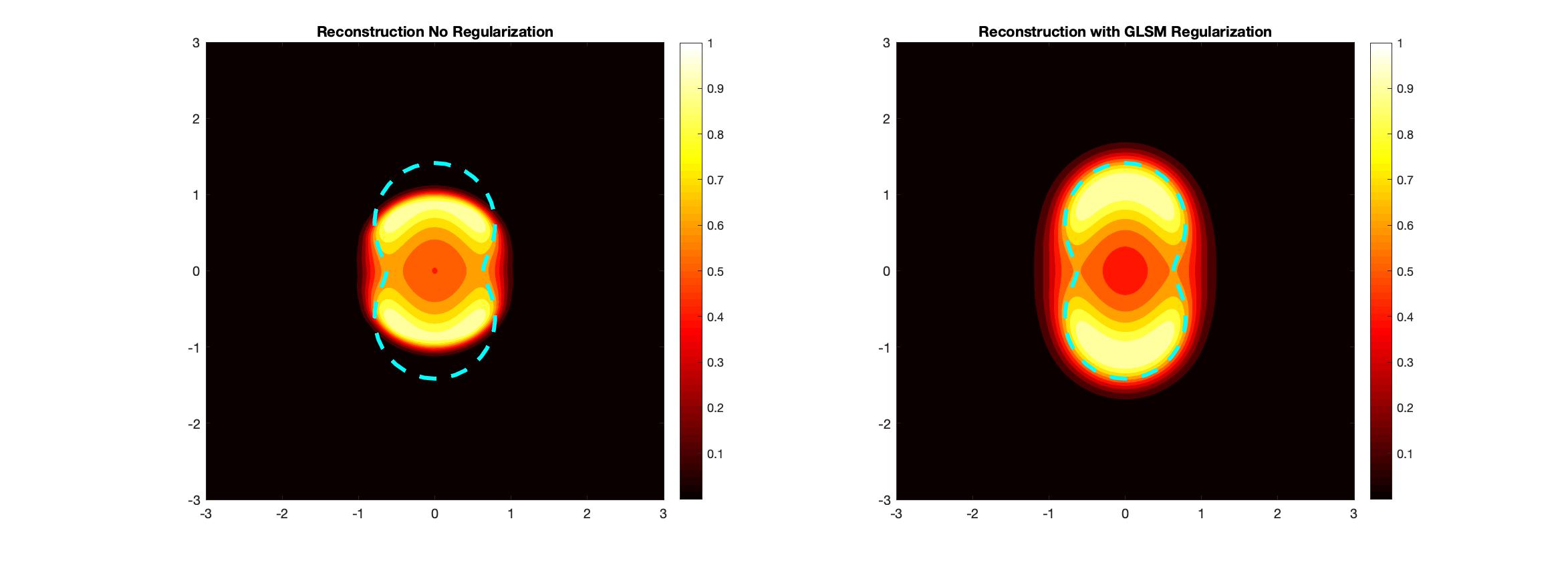}
\caption{Reconstruction of the star--shaped scatterer with and without regularization where we add $0\%$ random noise to the data. Here we use the Tikhonov filter function given in \eqref{filters} with $\alpha = 0.0001$. Left: reconstruction without regularization and Right: reconstruction with regularization.}
\label{recon3}
\end{figure}

We now check the restriction that $-k^2$ is not a Dirichlet eigenvalue for $D$. In Figure \ref{ReconDE}, we provide the numerical reconstruction for the kite--shaped scatterer with $k= 2.209855$, which is approximately the first Dirichlet eigenvalue for the scatterer. Here, the eigenvalue is computed using a double--layer ansatz. Only the boundary of the scatterer is being reconstructed, both inside and outside, we have small values of the imaging functional.
\begin{figure}[h]
\centering 
\includegraphics[scale=0.16]{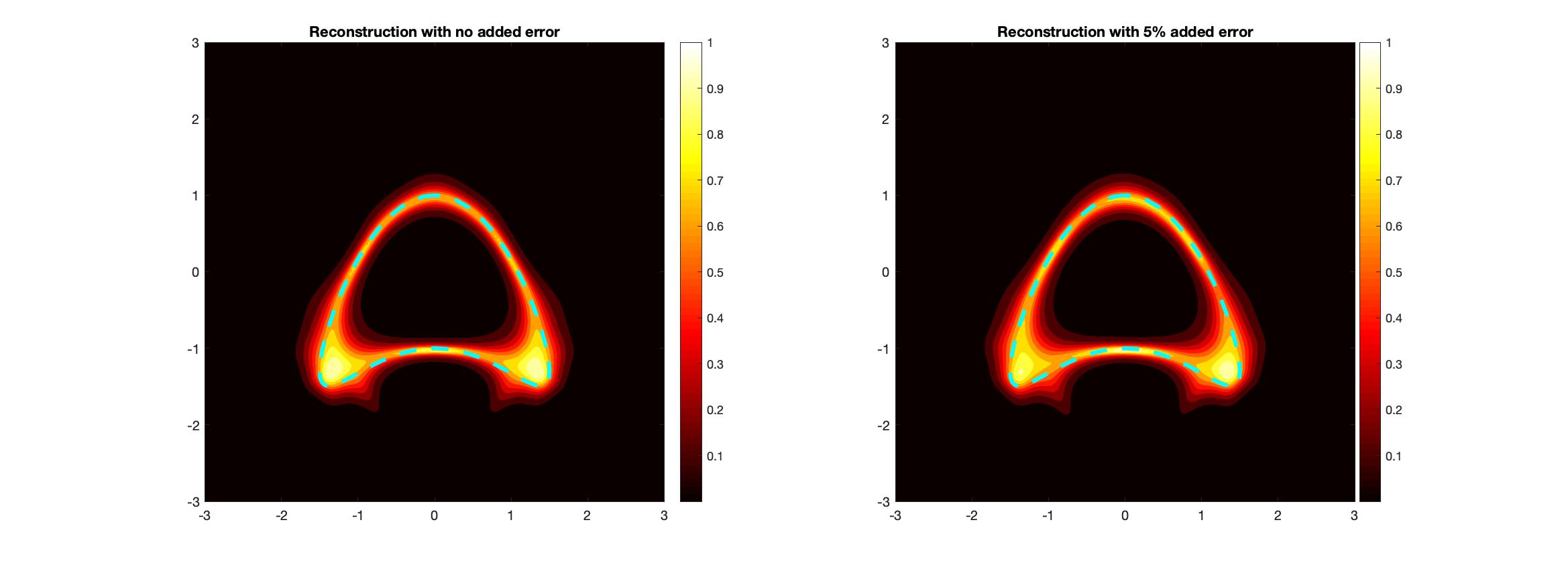}
\caption{Reconstruction of the kite--shaped scatterer with $k= 2.209855$ i.e. the first Dirichlet eigenvalue of the region. Here we use the GLSM filter function given in \eqref{filters} with $\alpha = 0.0001$. Left: reconstruction without added noise and Right: reconstruction with 5$\%$ added noise.}
\label{ReconDE}
\end{figure}

\subsubsection{Using different regularization parameters} 

Now, in this set of examples we wish to numerically test how the reconstruction depends on the regularization parameter $\alpha$. As in the previous examples, we will use an ad--hoc regularization parameter $\alpha=0.0001$. In Figure \ref{recon4}, we provide reconstructions of the kite--shaped scatterer using the GLSM and Spectral cut--off filter functions given in \eqref{filters}. In both reconstructions, $\delta=0.05$ we see that the reconstructions are comparable to the previous examples given.

\begin{figure}[h]
\centering 
\includegraphics[scale=0.16]{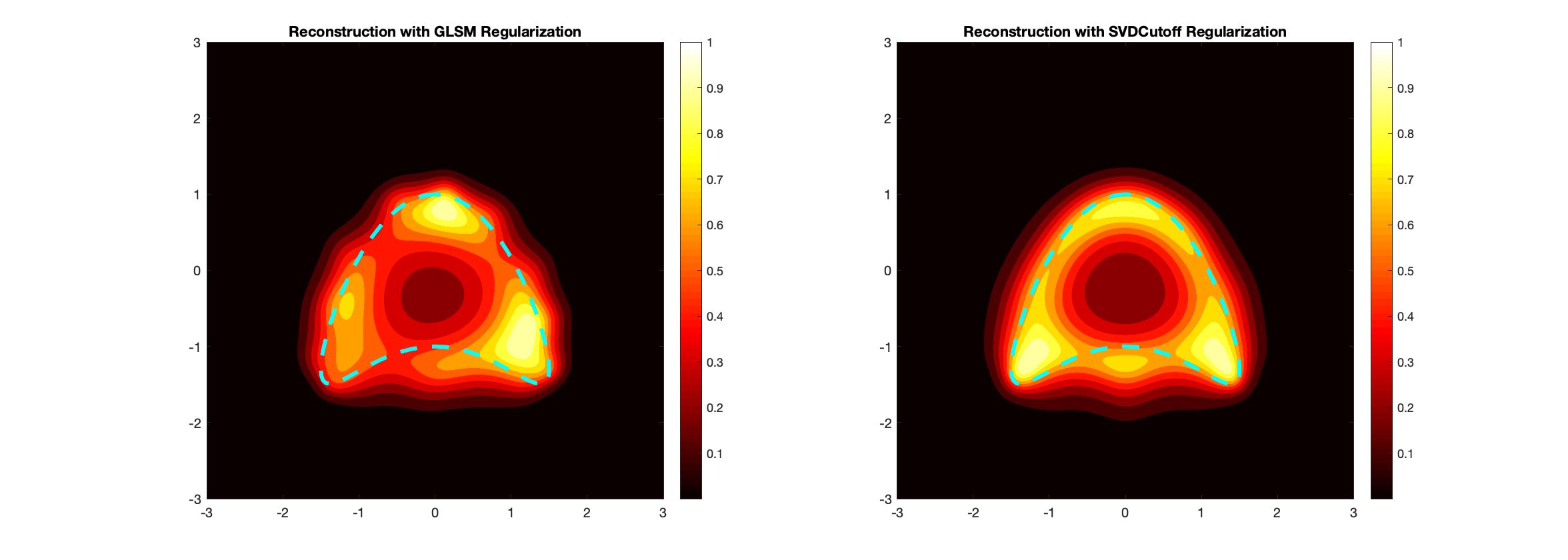}
\caption{Reconstruction of the kite--shaped scatterer using the GLSM and Spectral cut--off filter functions given in \eqref{filters} with $\alpha = 0.0001$. Here we add $5\%$ random noise to the data. Left: reconstruction  with GLSM regularization and Right: reconstruction with Spectral cut--off regularization.}
\label{recon4}
\end{figure}

In the following examples, we test the sensitivity of the Spectral cut--off regularization with respect to the regularization parameter. To this end, in Figure \ref{recon5} we provide reconstructions of the kite--shaped scatterer with $\delta=0.05$ for different regularization parameters $\alpha$. From this we see that the reconstructions are very similar for our two choices of the parameter $\alpha$. Here we consider the case when $\alpha=0.01$ and $0.0001$ in our reconstructions. Also, in Figure \ref{recon6} we provide a similar example for the star--shaped scatterer with $\delta=0.02$.

\begin{figure}[h]
\centering 
\includegraphics[scale=0.16]{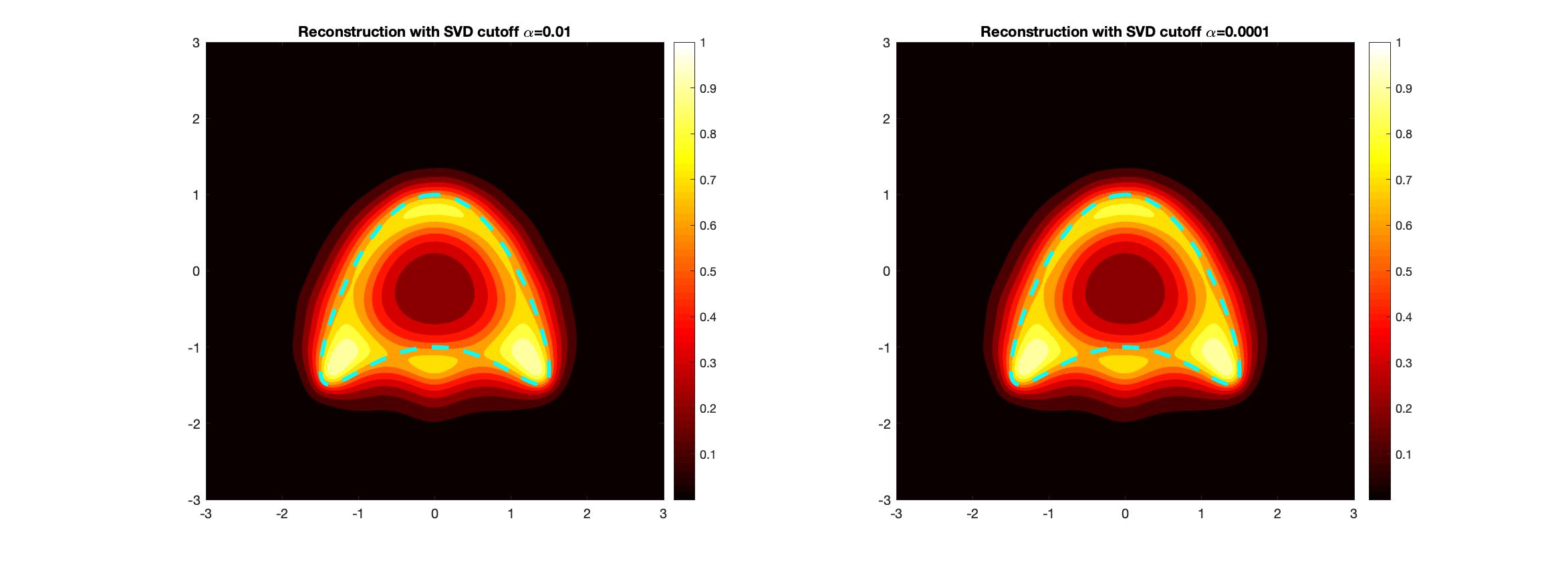}
\caption{Reconstruction of the kite--shaped scatterer using the Spectral cut--off filter functions given in \eqref{filters} for different values of $\alpha$. Here we add $5\%$ random noise to the data. Left: reconstruction with $\alpha =0.01$ and Right: reconstruction with $\alpha = 0.0001$.}
\label{recon5}
\end{figure}

\begin{figure}[h]
\centering 
\includegraphics[scale=0.16]{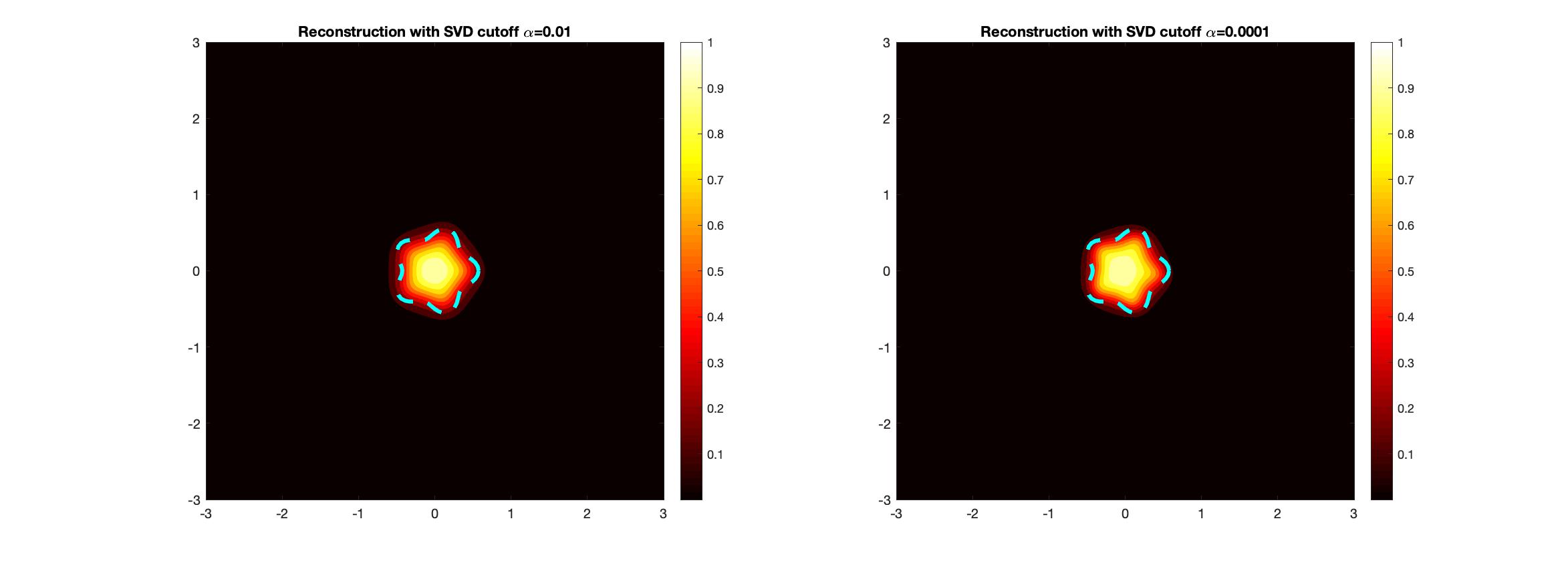}
\caption{Reconstruction of the star--shaped scatterer using the Spectral cut--off filter functions given in \eqref{filters} for different values of $\alpha$. Here we add $2\%$ random noise to the data. Left: reconstruction with $\alpha =0.01$ and Right: reconstruction with $\alpha = 0.0001$.}
\label{recon6}
\end{figure}

Even though the Spectral cut--off seems to provide similar reconstructions for different regularization parameters, this may not be true in general. In order to provide an analytical method for choosing a regularization parameter, in \cite{regfm2} a `discrepancy' principle was given, provided the noise level $\delta$ is known. Indeed, for the Tikhonov and GLSM regularization methods it is given that 
\begin{align} \label{reg-param}
\alpha_{\text{Tik}}(\delta) =  \frac{1}{4} \delta^{\left(\frac{1}{4}-p\right)} \quad \text{ and } \quad \alpha_{\text{GLSM}}(\delta) =  \delta^{\frac{1}{2}\left(\frac{1}{4}-p\right)}  \quad \text{for any $\,\,p\in (0,1/4)$.}
\end{align}
Since $\delta$ is not known explicitly in many applications, we approximate it via 
$$ \text{Error}= \| {\bf N} - {\bf N}^\top \|_2 /\|  {\bf N} \|_2 .$$ 
This is a good approximation of the noise level since the reciprocity relationship in \eqref{recip} implies that ${\bf N}$ is a symmetric matrix. In Figure \ref{recon7}, we provide the reconstruction of the peanut--shaped scatterer with $\delta=0.1$ for $\alpha$ chosen ad--hoc and optimally via \eqref{reg-param} for the Tikhonov filter function. Here we see that the reconstruction using the optimal regularization parameter gives a better image. 

\begin{figure}[h]
\centering 
\includegraphics[scale=0.16]{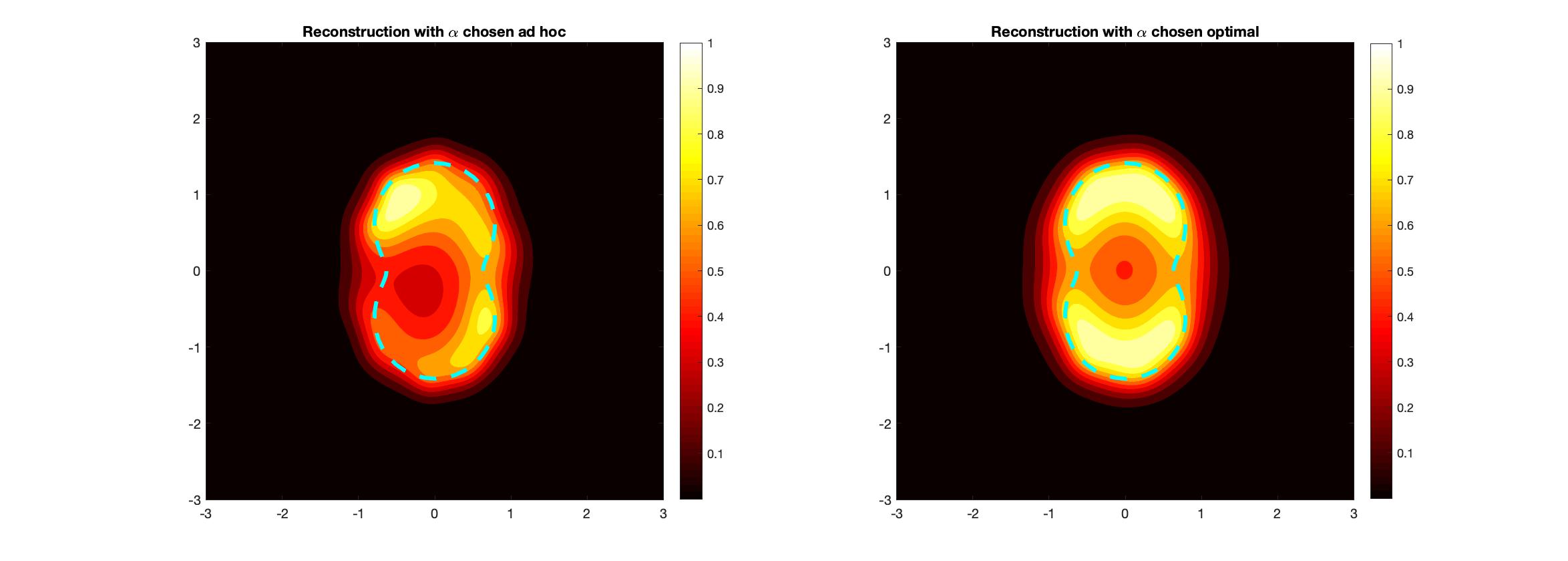}
\caption{Reconstruction of the peanut--shaped scatterer using the Tikhonov filter function given in \eqref{filters} with $\alpha$ chosen ad--hoc and chosen {optimally} where we add $10\%$ random noise to the data. Left: reconstruction using $\alpha = 0.001$  and Right: reconstruction using $\alpha = 0.25\cdot\text{Error}^{1/8}$.}
\label{recon7}
\end{figure}

\subsubsection{Reconstructions without Laplacian data} 

In this section, we consider the case when one only has access to the scattered field and not its Laplacian. Here, we wish to see if our aforementioned reconstruction method works in this case. Therefore, we recall that our near--field operator is actually defined by the kernel function $-2k^2 u_{\text{pr}} (x,y)$. Assuming that one only has the measured scattered field $u^{\text{scat}}$, we can use the fact that $u^{\text{scat}}= u_{\text{pr}} + u_{\text{ev}}$ along with the asymptotics 
$$ u_{\text{ev}} = \mathcal{O} \left({\text{e}^{-kr}}/{\sqrt{r}}\right) \quad \text{ as $\,\,\,r=|x| \to \infty.$}$$
To numerically see this, in Table \ref{small-ev} we provide the $\infty$--norms of the matrices
$${\bf U}_{\text{pr}}= \big[ \Delta_x u^{\text{scat}} (x_i , y_j) - k^2 u^{\text{scat}} (x_i , y_j) \big]_{i,j = 1}^{64}$$
\text{and }
$${\bf U}_{\text{ev}}= \big[ \Delta_x u^{\text{scat}} (x_i , y_j) + k^2 u^{\text{scat}} (x_i , y_j) \big]_{i,j = 1}^{64}$$
for our three scatterers as well as different collection curve radii $R$. These matrices are an approximation of the propagating and evanescent part of the scattered field.
\begin{table}[H]
	\centering
	\begin{tabular}{ c  l c l  c | c  }
		Scatterer     &  Radius & $\|  \bf{U}_{\text{pr}} \|_{\infty}$  &  $\| \bf{U}_{\text{ev}} \|_{\infty}$ \\
\hline 
\hline
		\vspace{-2.0ex}\\
		Kite--shaped  & $R=3$ & 0.1219 &  0.0012
		\vspace{1.5ex} \\ 
		Star--shaped  & $R=3$  & 0.0374  & 9.1526$\times10^{-7}$
		\vspace{1.5ex}\\
		Peanut--shaped   & $R=3$   & 0.0612  & 8.7533$\times10^{-5}$
		\vspace{1.5ex}\\ 
		Peanut--shaped   & $R=5$   & 0.0442  & 1.4755$\times10^{-8}$
		\vspace{1.5ex}\\ 
		Peanut--shaped   & $R=8$   & 0.0306  & 5.2090$\times10^{-14}$
		\vspace{1.5ex}\\ 
\hline
\end{tabular}\caption{In this table, we see that the evanescent part of the scattered field is small compared to the propagating field. } \label{small-ev}
\end{table}

Notice that for all three scatterers the evanescent part of the scattered field is significantly smaller than the propagating part. This implies that the near--field matrix can be approximated by 
$${\bf N}^{\text{scat}} = -2k^2\big[ u^{\text{scat},\delta} (x_i , y_j) \big]_{i,j = 1}^{64} \quad \text{ such that } \quad {\bf N} = {\bf N}^{\text{scat}}  + \mathcal{O} \left({\text{e}^{-kR}}/{\sqrt{R}}\right)$$
when the measurement curve has a radius of $R$. With this, just as in the so--called modified/approximate factorization method \cite{approxFM1,approxFM2,approxFM3} we expect that our reconstruction method will work if we use the eigenvalues and vectors of the matrix $ {\bf Q}{\bf N}^{\text{scat}} {\bf Q}^\top {\bf R}$ in our imaging functional $W_{\text{Reg}}(z)$. A similar idea was used in \cite{near-lsmBH} when the linear sampling method was applied to recovering a clamped obstacle.

Now, we present numerical examples for recovering the scatterer using only the scattered field. Here, we will provide examples for the kite and peanut--shaped scatterers. In Figure \ref{recon8} we give the reconstruction for the kite--shaped scatterer, where we use both pieces of scattering data $u^{\text{scat}}$ and $\Delta_x u^{\text{scat}}$ as well as the reconstruction using only the scattered field. In this example, the regularization parameter is chosen ad--hoc. Next, in Figure \ref{recon9} we recover the peanut--shaped scatterer using only the scattered field, where the regularization parameter is chosen optimally using \eqref{reg-param}. Here, the noise level $\delta$ is approximated just as in the previous section by the relative error of the matrix ${\bf N}^{\text{scat}}$ deviating from symmetry. Again, we check the applicability of the reconstructions providing an example with and without added noise to the data. Notice that in both examples that the reconstruction using just the scattered field is comparable to when both pieces of scattering data is known. 

\begin{figure}[h]
\centering 
\includegraphics[scale=0.16]{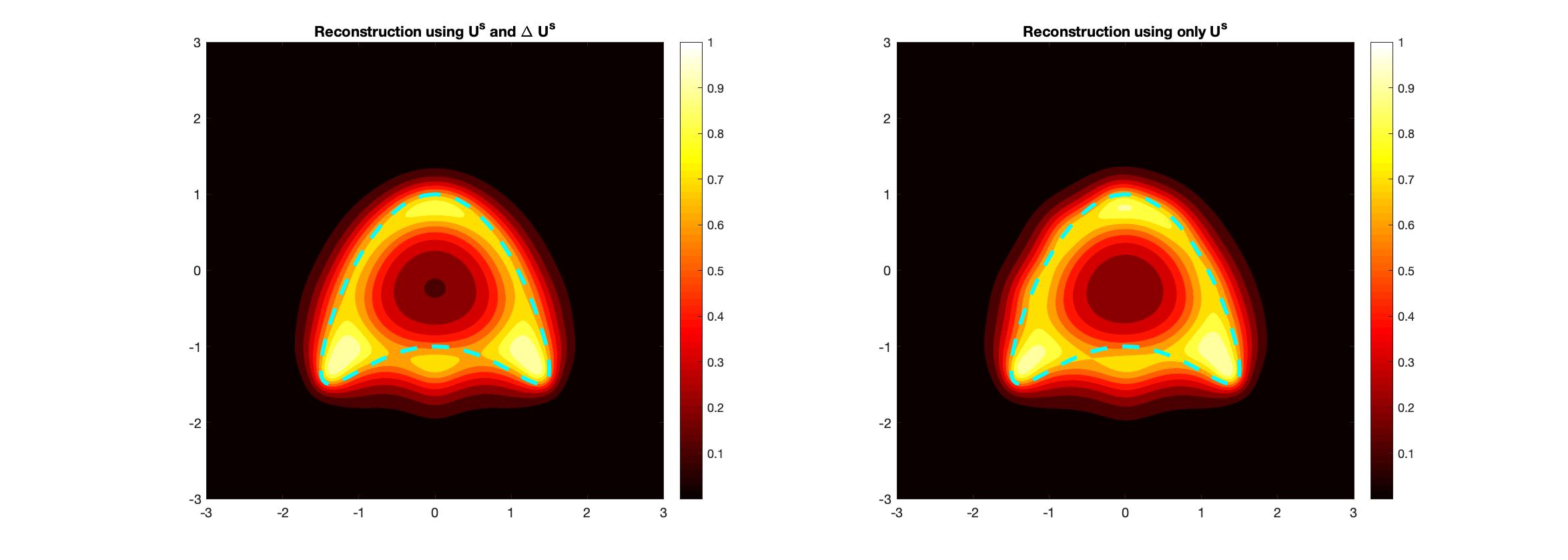}
\caption{Reconstruction of the kite--shaped scatterer using the Tikhonov filter function given in \eqref{filters} with $\alpha = 0.0001$ where we add $5\%$ random noise to the data. Left: reconstruction using ${\bf N}$ and Right: reconstruction using ${\bf N}^{\text{scat}}$.}
\label{recon8}
\end{figure}

\begin{figure}[h]
\centering 
\includegraphics[scale=0.16]{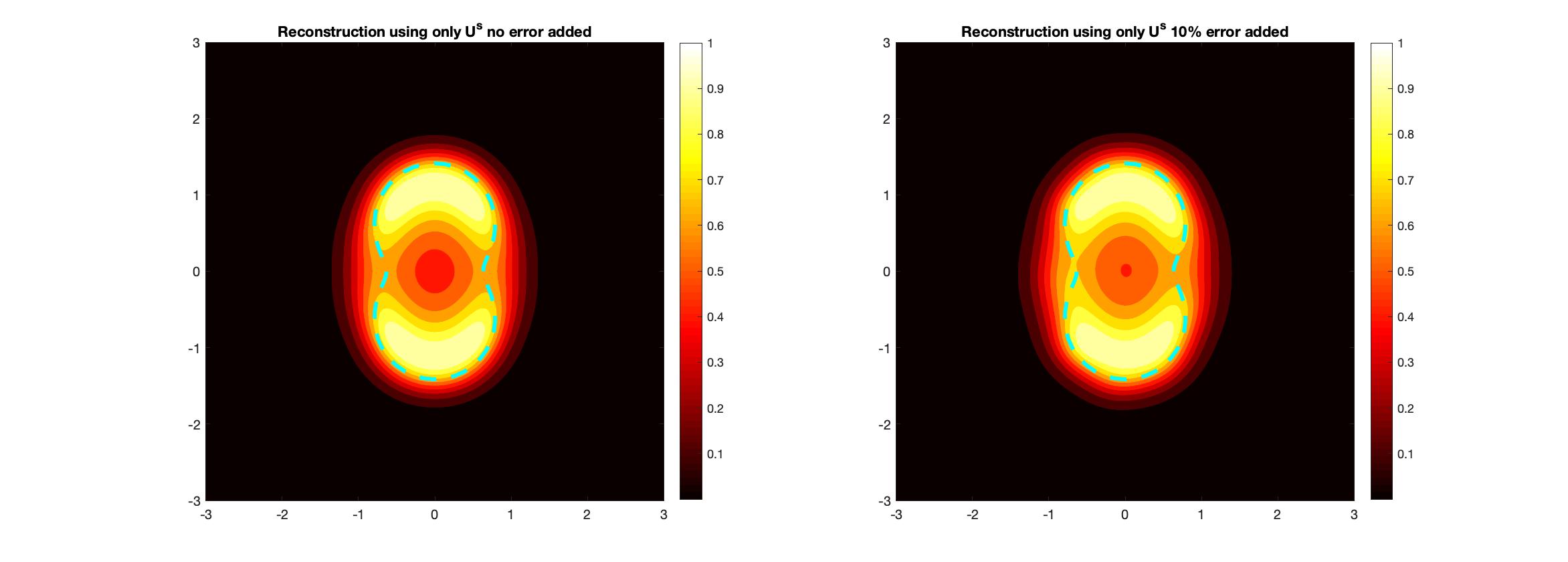}
\caption{Reconstruction of the peanut--shaped scatterer using the GLSM filter function given in \eqref{filters}, using only the scattered field with and without added noise and $\alpha=\text{Error}^{1/16}$. Left: reconstruction with no added noise and Right: reconstruction with 10$\%$ added noise.}
\label{recon9}
\end{figure}

\section{Conclusion}\label{conclusion}
In conclusion, we studied an inverse shape problem for recovering a simply supported obstacle from near--field point--source measurements for the biharmonic Helmholtz equation with applications to flexural wave scattering. Even though the analysis and computations are presented in $\R^2$ the same results would hold in $\R^3$. To apply the factorization method, we introduced a far--field transformation that augments the near--field operator into an operator admitting a symmetric factorization. This yields a rigorous characterization of the obstacle in terms of spectral data of the transformed operator and leads to a practical reconstruction algorithm via a regularized imaging functional. The numerical experiments with synthetic data demonstrate that the proposed post--processing produces stable reconstructions in the presence of noisy data. We provided extensive numerical examples to highlight the impact of regularization methods and parameter choice. We also investigated the case when only the scattered field (without Laplacian data) is available and observed that accurate reconstructions remain feasible when the measurement curve is sufficiently large.

Several extensions of this work are of interest. These include developing reconstruction procedures for other plate boundary conditions and material parameters, and studying limited--aperture measurement configurations. Indeed, here we take the simply supported plate boundary conditions assuming the Poisson ratio equal to one but in many applications (see \cite{ffdata-simpsupport}) this is not the case. Note that we have proven that the aforementioned near--field data uniquely determines the scatterer, but the factorization method studied here adds the assumption that $-k^2$ is not an interior Dirichlet eigenvalue of the scatterer. We found that this restriction on the wave number may not be necessary for a decent reconstruction, but it would be advantageous to derive a reconstruction method without it.

\vspace{0.2in}

\noindent{\bf Acknowledgments:} The research of author I. Harris is partially supported by the NSF DMS Grants 2509722 and 2208256.


\end{document}